\documentclass[a4paper,11pt,reqno]{amsart}

\usepackage[margin=1in,includehead,includefoot]{geometry}
\usepackage[utf8]{inputenc}
\usepackage[T1]{fontenc}
\usepackage{lmodern,microtype}
\usepackage{upref,stmaryrd }
\usepackage{mathtools,amssymb,amsthm,amsmath}
\usepackage[foot]{amsaddr}
\usepackage{enumitem}
\usepackage{todonotes}
\usepackage{graphicx}
\usepackage{hyperref}
\usepackage{wrapfig}
\usepackage{dsfont}
\usepackage{xcolor}
\usepackage{array}
\usepackage[margin=1cm]{caption}
\usepackage{bookmark}
\usepackage{mycommands}
\usepackage{nicefrac}
\graphicspath{{./graphics/}}

\makeatletter
\let\old@setaddresses\@setaddresses
\def\@setaddresses{\bigskip{\parindent 0pt\let\scshape\relax\let\ttfamily\relax\old@setaddresses}}
\makeatother

\newcommand{\moment}{\mathfrak{M}}
\newcommand{\momenttwo}{\mathfrak{M}^{(2)}}
\newcommand{\nf}[2]{\nicefrac{#1}{#2}}
\newcommand{\alives}{\mathcal N}

\newtheorem{thm}{Theorem}[section]
\newtheorem{corollary}[thm]{Corollary}
\newtheorem{proposition}[thm]{Proposition}
\newtheorem{lemma}[thm]{Lemma}
\newtheorem{assu}{Assumption}
\theoremstyle{remark}

\newtheorem{remark}{Remark}
\theoremstyle{definition}


\hypersetup{
  pdftitle={}
  pdfauthor={}
}

\begin{document}
\title[CLT for branching processes]{Central limit theorems for branching processes under mild assumptions on the mean semigroup}
\author{Bertrand Cloez$^1$ and Nicolás Zalduendo$^1$}
\footnotetext[1]{MISTEA, Université de Montpellier, INRAE, Institut Agro, Montpellier, France}


\begin{abstract}

We establish central limit theorems for a large class
of supercritical branching Markov processes in infinite dimension with
spatially dependent and non-necessarily local branching mechanisms.
This result relies on a fourth moment assumption and the exponential convergence of the mean semigroup in a weighted total variation norm. This latter assumption is pretty weak and does not necessitate symmetric properties or specific spectral knowledge on this semigroup.
In particular, we recover two of the three known regimes (namely the small and critical branching processes) of convergence in known cases, and extend them to a wider family of processes.
To prove our central limit theorems, we use the Stein's method,  which in addition allows us to newly provide a rate of convergence to this type of convergence.

\end{abstract}

\maketitle

\tableofcontents

\section{Introduction}

A branching Markov process on some state space $\mathcal X$, also called branching particle system, is heuristically defined as follows: we start at time $t=0$ with a particle evolving in $\mathcal X$ following a given Markovian dynamics; when the particle is located in $x\in \mathcal X$, it dies at rate $b(x)$, where $b:\mathcal X\longrightarrow \RR_+$ is a positive function; when it dies, the particle is replaced by a random number of descendants; the position of the new particles may depend on the location of their parent at the time of their death; finally, each new particle evolves independently following the same dynamics as their parent. See for instance  \cite{li2011measure} for an introduction to this model. 

We study such processes in the supercritical case; that is, when there is a positive probability for the process to survive indefinitely. Under appropriate moment conditions it is known that,  up to a scaling, the empirical measure of such particle system converges to some fixed measure (see for instance \cite{chen2007limit,englander2010strong,BansBerClo}). In this article, we are interested into studying the fluctuations around this limit, namely, the development of a central limit theorem associated to this law of large numbers result.  This question has already been addressed for the case where the state space $\mathcal X$ is a finite set, both in discrete time (see \cite{kesten1966limit,kesten1966additional}) and in continuous time (see \cite{athreya1969limit,athreya1969limitII,athreya1971some,janson2004functional}), where three different regimes of convergences where found depending on spectral properties of the mean semigroup of the process. These three regimes are called the small, critical and large branching cases, and a convergence of the fluctuations towards a Gaussian limit is proven in the two first cases.
Surprisingly, these results only have recently been extended to infinite dimension. In \cite{adamczak2015clt}, it is assumed that each particle evolves as an Ornstein-Uhlenbeck process between binary local branching events. There, the same three regimes of convergences where proven. In recent years, the extension of these result to more general frameworks has attracted a lot of attention. The first attempt to extent these result to general branching processes came with  \cite{ren2014central}, where a process where the reproduction occurs locally is considered. This means that the position of every descendant at the time of their birth is the same position where their predecessor died. In this setting, and under additional assumptions that imply that the infinitesimal generator of the process is a compact self-adjoint operator acting on $L^2(\mathcal X)$, the convergences on the three regimes was proven. By means of a Jordan decomposition of the infinitesimal generator, these results were then extended to the local non-symmetric case in \cite{ren2017central}. 
These setting exclude natural examples such as, for instance, age-structured models \cite{bansaye2020ergodic}, growth-fragmentation models \cite{bansaye2022non} or mutation type models \cite{cloez2024fast}.
Finally, not unrelated, we can cite \cite[Proposition 6.4]{BDMT} which describes a different central limit theorem in long time around fixed time intervals.

Using a different approach, we extend these results into a more general framework including such examples. Indeed, we prove Gaussian fluctuations for this type of processes in case when it is expected to hold.

This paper is organised as follows: in Section~\ref{sec:def} we formally define the process, we present the main assumptions and we state the main results of this article (Theorems~\ref{thm:CLT_Martingale}, \ref{thm:CLT_SBC} and \ref{thm:CLT_CBC}), which are proved in Section~\ref{sec:main_proof}. The proof of our results is based in Stein's method, for which we give a brief overview in Subsection~\ref{sec:Stein}. Prior to that, in Section~\ref{sec:martingale} we study the asymptotic behaviour of the moments of the process and prove a law of large numbers. Finally, in Section~\ref{sec:applications}, we apply our results to the \textit{house of cards} model.

Throughout this article, we assume that all the random variables that we use are defined on a given probability space $(\Omega, \mathcal F, \mathbb P)$ and we denote by $\mathbb E$ the expectation associated with $\mathbb P$. The set $\NN$ will denote the set of non-negative integers, and $\NN^*=\NN\setminus \{0\}$. For $p\in \NN^*$, we denote by $L^p(\PP)$ the space of $p-$integrable random variables with respect to $\PP$. 
\section{Assumptions and main results}\label{sec:def}

In this work, we consider a branching Markov process $(Z_t)_{t\geq 0}$  on a measure space $(X,B(\mathcal X), \mu)$. This correspond to a point measure valued process satisfying the branching property. Namely, on the first hand, we  have
\[\forall t\geq 0, \, Z_t := \sum_{u\in \alives_t} \delta_{X_t^u},\]
where $\mathcal N_t$ stands for the individuals $u\in \mathcal U$ who are alive at time $t$, and $X_t^u\in \mathcal X$ stands for the trait of the individual $u \in \mathcal N_t$. The genealogy (namely a precise description of sets $\mathcal N_t$) can be encoded by the well-known Ulam notation; see for instance \cite{le2005random} for details. On the second hand, process $(Z_t)_{t\geq 0}$ verifies the branching property. This reads
\[ \forall t,s\geq 0, \quad Z_{t+s} = \sum_{u\in \alives_t} Z_{t,t+s}^u = \sum_{u\in \alives_t}\underbrace{\sum_{v\in \alives_{t,t+s}^u}\delta_{X_{t+s}^v}}_{Z_{t,t+s}^u},\]
where processes $(Z_{t,t+s}^{u})_{s\geq 0}$ are independent conditioned on $\mathcal F_t :=\sigma (Z_r,\ r\leq t)$. From a modelling point of view, $\alives_{t,t+s}^u$ stands for the individuals who are descendants of $u$ (itself included).

 This branching property implies that we can assume that the process is initiated with only one particle (namely $Z_0$ is a Dirac mass). In this case, for $x\in \mathcal X$,  we write $\mathbb P_x$ for the law of the process conditioned on $\{Z_0=\delta_x\}$. We also write $\mathbb E_x$ for the expectation associated to $\PP_x$.

Our main results rely on two different types of assumptions, an assumption of moments and an ergodicity assumption. Both of these assumptions will depend on a function $V: \mathcal X \longrightarrow (0,+\infty)$, that will be fixed for the rest of the paper. With the use of this function we also define for $k\in \NN^*$
\[\mathcal B(V^k) := \left\{f:\mathcal X \longrightarrow \RR,\ \sup_{x\in \mathcal X} \dfrac{|f(x)|}{V^k(x)}<+\infty\right\},\]
and for $f\in \mathcal B(V^k)$, we set
\[\|f\|_{V^k} :=\sup_{x\in \mathcal X} \dfrac{|f(x)|}{V^k(x)}.\]
We also denote $\mathcal M(V^k)$ (resp. $\mathcal M_+(V^k)$) the space of (signed) measures (resp. positive measures) $\mu$ defined on $B(\mathcal X)$ such that $\mu(V^k):= \int_{\mathcal X}V^kd\mu$ is finite.

We begin by an assumption on moment of order $\kappa \in \mathbb N^*$, which depending of the results, can be applied to different values of $\kappa$. Moreover, this assumption is local but may be extended to more general time with Proposition~\ref{prop:moments}.

\begin{assu}\label{assu:moments}
For some $\kappa \in \NN^*$, there exist $t^*>0$ and a constant $a_1>0$ such that for all $x\in \mathcal X$ and $k\leq \kappa$,
\[ \sup_{t \leq t^*} \mathbb E_x\left(Z_t(V)^k\right) \leq a_1V^k(x).\]
\end{assu}


Assumption~\ref{assu:moments} may be verified using simple drift conditions as \cite[Theorem 2.1]{meyn1993stability}. See for instance \cite{BansBerClo} in the context of branching processes.

When Assumption~\ref{assu:moments} holds we define, for $k\leq \kappa$, the family of operators $(\moment_t^{(k)})_{t\geq 0}$ given for $f\in \mathcal B(V)$ and $x\in \mathcal X$ by
\[\moment_t^{(k)}[f](x) = \mathbb E_x(Z_t(f)^k).\]


All along the article, we will assume that Assumption~\ref{assu:moments} holds for $\kappa=1$ and simplify the notation $\moment_t f:= \moment^{(1)}_t[f]$. We then have that $t\mapsto \moment_t V$ is finite and locally bounded in a neighbourhood of $t=0$. This condition, and the branching property, ensures that $(\moment_t)_{t\geq 0}$ forms a strongly continuous semigroup on $\mathcal B(V)$. We will assume geometric ergodicity on this semigroup. In order to state this assumption, we define for $\kappa \in \NN^*$ the functions 
\[\underline{V}^{(\kappa)}(x) = \min_{k\leq \kappa} V^k(x)\text{ and }\overline{V}^{(\kappa)}(x) = \max_{k\leq \kappa}V^k(x)\]

\begin{assu}\label{assu:eigenelements}
For some $\kappa \in \NN$, there exists a triplet $(\gamma, h , \lambda)\in \mathcal M_+(\overline{V}^{(\kappa)})\times \mathcal B_+(\underline{V}^{(\kappa)})\times \RR^*_+$ of eigenelements of the semigroup $\moment_t$, with $\gamma (h) = \|h\|_{V}=1$, a positive constant $a_2>0$ and $\rho>0$, satisfying for any $x\in \mathcal X$, all $k\leq \kappa$, any $f\in \mathcal B(V^k)$ and $t>0$,
\begin{equation*}
|e^{-\lambda t} \moment_t f(x) -\gamma(f)h(x)|\leq a_2e^{-\rho t}\|f\|_{V^k} V^k(x). 
\end{equation*}
\end{assu}

Assumption~\ref{assu:eigenelements} is a pretty mild assumption on the semigroup convergence. Our approach therefore covers the models in the articles \cite{bansaye2020ergodic, bansaye2022non, champagnat2023general, del2002stability,ferre2021more,sanchez2023krein, meyn1993stability,canizo2023harris, bertoin2019feynman, hairer2011yet} and on their numerous extensions from several authors.


We can now state a law of large numbers. This result has been proved under different assumptions in the literature (see \cite{chen2007limit,englander2010strong,BansBerClo})

\begin{thm}[Law of Large Numbers] \label{thm:LGN}
If Assumptions~\ref{assu:moments} and~\ref{assu:eigenelements} hold for some even $\kappa\in \NN^*$, then there exists a non-negative random variable $W$ such that for all $f\in \mathcal B(V)$, and for all $x\in\mathcal X$
\[e^{-\lambda t} Z_t(f) \xrightarrow[]{t\to +\infty}\gamma(f) W\text{ in }L^\kappa(\mathbb P_x).\]
\end{thm}

The main result of this article is the study of the fluctuations of the process around the random limit provided by Theorem~\ref{thm:LGN}. To do so, we start by recalling that a distance between the laws of two real random variables $X,Y$ can be expressed as
\begin{equation}
d_{\mathcal F}(X,Y) = \sup_{F\in \mathcal F}|\mathbb E(F(X)) - \mathbb E(F(Y))|,
\end{equation}
where $\mathcal F$ is a suitable family of test functions. In what follows, we will work with the family 
\[\mathcal F:=\left\{F:\RR \longrightarrow \RR,\ \max\{\|F\|_\infty, \|F'\|_\infty,\|F''\|_\infty,\|F'''\|_\infty\}\leq 1\right\},\]
and we will denote by $\textbf{d}$ the associated distance. One can easily check that if a sequence of random variables $(X_t)_{t\geq 0}$ is convergent in distance $\textbf{d}$, then it converges in law to the same limit (see for instance \cite[Lemma 5.1]{benaim2017ergodicity}).

\begin{thm}\label{thm:CLT_Martingale} If Assumptions~\ref{assu:moments} and~\ref{assu:eigenelements} hold for $\kappa \geq 4$, then there exists $\sigma_h>0$ such that
\[e^{-\nf{\lambda}{2}}(Z_t(h)-e^{\lambda t}W) \longrightarrow \sigma_h \sqrt{W}\mathcal Z\]
in distribution when $t$ goes to infinity, where $\mathcal Z$ is a standard Gaussian random variable that is independent of $W$. Moreover, we have the following estimate for the speed of convergence in distance \textbf{d}: there exists $K>0$ such that for all $t>0$ 
    \[\textbf{d}\left(e^{-\nf{\lambda t}{2}}(Z_t(h)-e^{\lambda t}W) , \sigma_{h} \sqrt{W}\mathcal Z\right)\leq K\exp\left(-\dfrac{\lambda \rho}{2\rho + \lambda }\right)\mathbb E(Z_0(V^*)),\]
    with $V^*(x)=\max\{\sqrt{V(x)},V^3(x)\}$. 

\end{thm}
The previous theorem can be stated even under more weakly assumptions on the convergence of the semigroup (see Remark~\ref{rem:weaker_assumption} after the proof of Theorem~\ref{thm:CLT_Martingale} in Section~\ref{sec:main_proof} for more details). 
The behaviour of the fluctuations associated top the limit in Theorem\ref{thm:LGN} for a general function $f\in \mathcal B(V)$ depends on the relation between $\lambda$ and $\rho$ in Assumption~\ref{assu:eigenelements}. In the following result, we provide a speed of convergence for the case $2\rho > \lambda$, known in the literature as the small branching case.

\begin{thm}[Small branching case]\label{thm:CLT_SBC} If Assumptions~\ref{assu:moments} and ~\ref{assu:eigenelements} hold for $\kappa\geq 4$ and $2\rho> \lambda$, then for all $f\in \mathcal B(V)$ there exists $\sigma_{f,s}> 0$ such that
    \begin{equation*}
        e^{-\nf{\lambda t}{2}}(Z_t(f) -e^{\lambda t}\gamma(f)W) \longrightarrow \sigma_{f,s}\sqrt{W} \mathcal Z
    \end{equation*}
    in distribution when $t$ goes to infinity, where $\mathcal Z$ is a standard Gaussian random variable that is independent of $W$. Moreover, we have the following estimate for the speed of convergence in distance \textbf{d}: for all $f\in \mathcal B(V)$ there exists $K>0$ such that for all $t>0$ 
    \[\textbf{d}\left(e^{-\nf{\lambda t}{2}}(Z_t(f)-e^{\lambda t}\gamma (f)W) , \sigma_{f,s}^2 \sqrt{W}\mathcal Z_1\right)\leq K\exp\left(\dfrac{\lambda(\lambda - 2\rho)}{2(\lambda + 2\rho)}t\right)\mathbb E(Z_0(V^*)),\]
    with $V^*(x)=\max\{\sqrt{V(x)},V^3(x)\}$.
\end{thm}

For the case $2\rho = \lambda$, we provide in the following result a central limit theorem. However, our mild assumptions do not allow us to estimate a precise speed of convergence in this case.

\begin{thm}[Critical branching case]\label{thm:CLT_CBC}If Assumptions~\ref{assu:moments} and ~\ref{assu:eigenelements} hold for $\kappa\geq 4$ and $2\rho= \lambda$, then for all $f\in \mathcal B(V)$ there exists $\sigma_{f,c}\geq 0$ such that
    \begin{equation*}
       t^{-\nf{1}{2}} e^{-\nf{\lambda t}{2}}(Z_t(f) -e^{\lambda t}\gamma(f)W) \longrightarrow \sigma_{f,c}\sqrt{W} \mathcal Z
    \end{equation*}
    in distribution when $t$ goes to infinity, where $\mathcal Z$ is a standard Gaussian random variable that is independent of $W$.
\end{thm}
The proofs of Theorem~\ref{thm:CLT_Martingale}, Theorem~\ref{thm:CLT_SBC} and Theorem~\ref{thm:CLT_CBC} are based on Stein's method and are presented in Section~\ref{sec:main_proof}. We remark that we recover the same convergences for the small and critical branching cases that appear in the existing literature (see \cite{kesten1966limit,adamczak2015clt,ren2014central,ren2017central}). However, our assumptions do not allow us to cover the large branching case $\lambda> 2\rho$, which is suspected to correspond to a non-Gaussian and almost-sure or $L^2$ limit of the fluctuations, as it is the case of the finite dimensional setting \cite{kesten1966limit,athreya1969limit} and diagonalizable/jordanizable setting (see \cite{adamczak2015clt,ren2014central} or \cite[Theorem 1.14]{ren2017central}]). 

Note that we have presented Theorem~\ref{thm:CLT_SBC} and Theorem~\ref{thm:CLT_CBC} in the simple form of the convergence of the fluctuation to a normal law (with an associated rate of convergence) but our approach naturally allows to recover the side results of \cite{ren2017central}. More precisely, the explicit expression of the variance we obtain permits to rewrite the proof of \cite[Corollary1.17]{ren2017central} and to obtain the (functional) convergence of the fluctuation process towards a Gaussian fields with explicit covariance matrix. We can also have joint convergence of the fluctuation and quantities as $e^{-\lambda t} Z_t(f)$ through some Slutsky's Lemma type arguments.

The values of $\sigma_h^2$, $\sigma_{f,s}^2$ and $\sigma_{f,c}^2$ correspond to the limit of the properly rescaled variance of $(Z_t - e^{\lambda t} \gamma(f) W)_{t\geq 0}$ (see Corollary~\ref{cor:fluctuations_order2} below). We notice that the limiting Gaussian random variable in the small branching case is non-degenerate when the process is not trivial (see Remark~\ref{rem:non_degenerate}). However, in the critical branching case, Theorem~\ref{thm:CLT_Martingale} implies that $\sigma_{h,c}^2=0$ and thus the limit is not always non-degenerate.

Note that the Article~\cite{Emma} , written independently and in parallel to our work, demonstrates three convergence regimes with an additional hypothesis concerning the second eigenfunction of the mean semigroup (excluding in particular, as \cite{ren2017central,ren2014central}, examples such as the \textit{house of cards} model developed in Section~\ref{sec:applications}). 


\section{Properties of the intrinsic martingale}\label{sec:martingale}

As usual, setting $W_t:=e^{-\lambda t}Z_t(h)$, where $h\in \mathcal B(V)$ is the eigenfunction given by Assumption~\ref{assu:eigenelements}, then $(W_t)_{t\geq 0}$ forms a non-negative martingale. 

We stress the fact that we are assuming $\lambda >0$. In addition, and for convenience, for the rest of the article we will always assume without loss of generality that $\lambda > \rho$ in Assumption~\ref{assu:eigenelements}. In fact, if $\lambda \leq \rho$, then we can always find $\widehat \rho <\lambda$ such that Assumption~\ref{assu:eigenelements} holds for $\widehat \rho$.
\subsection{Asymptotic behaviour of moments and proof of Theorem~\ref{thm:LGN}}\label{sec:moments} We start with a result concerning the moments of the process.

\begin{proposition}\label{prop:moments}
If Assumptions~\ref{assu:moments} and \ref{assu:eigenelements} hold for some $\kappa\in \NN^*$, then there exists $C>0$ such that for all $f\in \mathcal B(V)$, $k\leq \kappa$ and $x\in \mathcal X$,
\begin{equation}\label{eq:general_bound}
    \forall t\geq 0,\ | \moment_t^{(k)}[f](x) |\leq C \|f\|_V^k e^{\lambda k t}V^k(x).
\end{equation}
Moreover when $\gamma(f)=0$, we have  
\begin{enumerate}
    \item (Large Branching Case) If $2\rho <\lambda$ then 
    \begin{equation*}
    \forall t\geq 0,\ | \moment_t^{(k)}[f](x) | \leq C \|f\|_V^k e^{(\lambda - \rho)kt} V^k(x).
    \end{equation*}
    \item (Small Branching Case) If $2\rho >\lambda$ then 
    \begin{equation*}
    \forall t\geq 0,\ | \moment_t^{(k)}[f](x) | \leq C \|f\|_V^k e^{\nf{\lambda kt}{2}} V^k(x).
    \end{equation*}
    \item (Critical Branching Case) If $2\rho = \lambda$, then
    \begin{equation*}
    \forall t\geq 1,\ |\moment_t^{(k)}[f](x)| \leq C\|f\|_V^k t^{ \lfloor\nf{k}{2}\rfloor } e^{\nf{\lambda k}{2}t}V^k(x).
    \end{equation*}
\end{enumerate}
\end{proposition}


\begin{proof}[Proof of Proposition~\ref{prop:moments}]
Firstly, by Assumption~\ref{assu:eigenelements} we have for any $f\in \mathcal B(V)$
\[ |\moment_t^{(1)}[f](x)| = |\moment_t f(x)|\leq a_2e^{(\lambda - \rho)t}\|f\|_VV(x) + e^{\lambda t}\gamma (f)h(x)\leq C_1\|f\|_Ve^{\lambda t}V(x),\]
where $C_1=(a_2+\gamma(V))$. If in addition $\gamma(f)=0$ we have $ \moment_tf(x)\leq a_2\|f\|_V e^{(\lambda -\rho)t}V(x)$. We then have the result for $k=1$ in the large, small and critical branching cases.  Let us show all possible cases for larger $k$ by an induction argument. In what follow, for sake of presentation, we will consider that $\| f\|_V=1$, which does not affect generality, since it suffices to consider $f/\|f\|_V$ instead of $f$. Let us assume that for some $ k\in  \llbracket 1, \kappa-1 \rrbracket$, we have
\begin{equation}
\label{eq:forme-general-rec}
\forall j \leq k, \quad |\moment_t^{(j)}[f](x)| \leq C_k  t^{\delta j} e^{\beta j t} V^j(x),
\end{equation}
for some constant $C_k>0$, $\delta \in\{0,\nf{1}{2}\}$ and $\beta\in \{\lambda, \lambda-\rho, \nf{\lambda}{2}  \}$ 

For a set $\alives$,  we introduce
\[A(\alives)=\left\{ \alpha\in \NN^\alives,\ \sum_{u\in \alives} \alpha_u=k+1\right\}, \ \overline{A}(\alives) = \{\alpha \in A,  \exists u \in \alives, \ \alpha_u=k+1 \}. \]
We recall that $t^*>0$ is defined in Assumption~\ref{assu:moments}. Fix $t_0\leq \min\{t^*,1\}$. For $r>0$ we have by using the Multinomial Theorem,
\begin{align*}
    \mathbb E_x (Z_{t_0+r}(f)^{k+1}) 
    &= \mathbb E_x\left(\left(\sum_{u\in \alives_{t_0}} Z_{t_0,t_0+r}^u(f)\right)^{k+1}\right)\\
    &= \mathbb E_x\left(\sum_{\alpha \in A(\alives_{t_0})} \frac{(k+1)!}{\prod_{u\in \alives_{t_0}} \alpha_u!} \prod_{u\in \alives_{t_0}} Z_{t_0,t_0+r}^u(f)^{\alpha_u}\right).\\
\end{align*}

In what follows, for sake of notation we set $A = A(\alives_{t_0})$ and $\overline{A}=\overline{A}(\alives_{t_0}).$ Taking conditional expectation with respect to $\mathcal F_{t_0}$ and using \eqref{eq:forme-general-rec} we obtain
\begin{align*} 
\mathbb E_x(Z_{t_0+r}(f)^{k+1})&=\mathbb E_x\left(\sum_{\alpha \in A} \frac{(k+1)!}{\prod_{u\in \alives_{t_0}} \alpha_u!} \prod_{\substack{u\in \alives_{t_0}\\ \alpha_u\neq 0}} \moment_r^{(\alpha_u)}[f](X^u_{t_0})\right)\\
    &=\mathbb E_x\left(\sum_{u\in \alives_{t_0}}  \moment_r^{(k+1)}[f](X^u_{t_0})\right) + \mathbb E_x\left(\sum_{\alpha \in A \setminus \overline{A}} \frac{(k+1)!}{\prod_{u\in \alives_{t_0}} \alpha_u!} \prod_{\substack{u\in \alives_{t_0}\\ \alpha_u\neq 0}} \moment_r^{(\alpha_u)}[f](X^u_{t_0})\right)\\
    &\leq \moment_{t_0} \moment_r^{(k+1)}[f](x) + \mathbb E_x\left(\sum_{\alpha \in A \setminus \overline{A}}  \frac{(k+1)!}{\prod_{u\in \alives_{t_0}} \alpha_u!} \prod_{\substack{u\in \alives_{t_0}\\ \alpha_u \neq 0}} C_k  r^{\delta \alpha_u} e^{\beta \alpha_u r} V^{\alpha_u}(X^u_{t_0})\right).
\end{align*}
Now, using Assumption~\ref{assu:moments} (since $t_0\leq t^*$), the expectation on the right-hand side is equal to
\begin{align*}
    C_k^{k+1}r^{\delta (k+1)} e^{\beta (k+1) r} &\mathbb E_x\left(\sum_{\alpha \in A \setminus \overline{A}} \frac{(k+1)!}{\prod_{u\in \alives_r} \alpha_u!}  \prod_{u\in \alives_r}  V^{\alpha_u}(X^u_r)\right)\\
    &\leq  C_k^{k+1}r^{\delta (k+1)} e^{\beta (k+1) r} \mathbb E_x\left(\left(\sum_{u\in \alives_r} V(X^u_r)\right)^{k+1}\right)\\
    &\leq C_{0,k+1}r^{\delta (k+1)} e^{\beta (k+1) r}  V^{k+1}(x)
\end{align*}
where $C_{0,k+1}=C_k^{k+1}a_1$. Hence, we have
\begin{align}
\label{eq:maj-mom-tmp}
    \moment_{t_0+r}^{(k+1)} [f](x)
    &\leq \moment_{t_0} \moment_r^{(k+1)}[f](x) + C_{0,k+1} r^{\delta (k+1)} e^{\beta (k+1) r} V^{k+1}(x).
\end{align}
Now consider $t>0$, then there exist a unique $n\in \NN$ and a unique $r\in [0,t_0)$ such that $t=nt_0+r$. Using this decomposition, a simple iteration of \eqref{eq:maj-mom-tmp} gives that 
\begin{align*}
    \moment_t^{(k+1)}[f](x)&=\moment_{n t_0 + r}^{(k+1)} [f](x)\\
    &\leq \moment_{n t_0} \moment_{r}^{(k+1)}[f](x) + \sum_{\ell = 0}^{n-1} C_{0,k+1} ((n-1-\ell)t_0+r)^{\delta (k+1)} e^{\beta (k+1) ((n-1-\ell)t_0+r)} \moment_{\ell t_0} V^{k+1}(x).\\
    &\leq a_1a_2e^{\lambda nt_0}V^{k+1}(x) + C_{0,k+1}a_2 t^{\delta (k+1)} e^{\beta (k+1)t} \sum_{\ell = 0}^{n-1}   e^{(\lambda - \beta (k+1))\ell t_0} V^{k+1}(x),
\end{align*}
where we used Assumptions~\ref{assu:moments} and~\ref{assu:eigenelements}. Then, if $\lambda \neq \beta (k+1)$, there exists $C_{1,k+1}>0$ such that
\begin{align*}
    \moment_t^{(k+1)} [f](x)
    & \leq a_1a_2e^{\lambda t}V^{k+1}(x) + C_{0,k+1} a_2 t^{\delta (k+1)}e^{\beta(k+1)r}\dfrac{e^{\lambda nt_0}-e^{\beta(k+1)nt_0}}{e^{(\lambda -\beta(k+1))t_0}-1}V^{k+1}(x)\\
    &\leq C_{1,k+1}\left(e^{\lambda t} + t^{\delta(k+1)}e^{\max\{\lambda, \beta(k+1)\}t}\right)V^{k+1}(x),
\end{align*}
which rewrites, for $t\geq t_0$ and some $C_{2,k+1}$,
\begin{equation}
\label{eq:conclusion-rec1}
\moment_t^{(k+1)} [f](x) \leq C_{2,k+1} t^{\delta (k+1)}  e^{\max\{\beta (k+1), \lambda\} t}V^{k+1}(x) 
\end{equation}
Similarly, in case where $\lambda = \beta (k+1)$, we have for some constant $C_{3,k+1}>0$,
\begin{equation}
\label{eq:conclusion-rec2}
\moment_t^{(k+1)} [f](x) \leq C_{3,k+1} t^{\delta (k+1)+1}  e^{\beta (k+1) t} V^{k+1}(x).
\end{equation}

Let us now conclude our induction argument by dealing with each of the different cases in our proposition. In the general case, $\delta=0$ and $\beta = \lambda$ then \eqref{eq:conclusion-rec1} holds (because $k+1>1$) and gives  \eqref{eq:general_bound}. Consider now the cases where $\gamma(f)=0$.  When $2\rho <\lambda$, we have $\delta=0$ and $\beta=\lambda -\rho$. Consequently $\beta (k+1) \geq 2\lambda -2\rho>\lambda$ and then we conclude using \eqref{eq:conclusion-rec1}.  The case where $2\rho >\lambda$ is more particular. For $k=1$,
\eqref{eq:forme-general-rec} holds with $\delta=0$ and $\beta=\lambda-\rho < \lambda/2$ then \eqref{eq:conclusion-rec1} applies and gives that \eqref{eq:forme-general-rec} holds for $k=2$ with $\delta=0$ and $\beta=\lambda/2$, and we can then iterate the argument. Finally, in the case where $2\rho =\lambda$, we have that for $k=1$,
\eqref{eq:forme-general-rec} holds with $\delta=0$ and $\beta=\lambda-\rho = \lambda/2$ then \eqref{eq:conclusion-rec2} applies and gives that for for $k=2$,
\eqref{eq:forme-general-rec} holds with $\delta=1/2$ and $\beta=\lambda/2$, we can then iterate since we will always have $\lambda < \beta (k+1)$, with $\beta=\lambda/2$. It ends the proof.

\end{proof}
\begin{proof}[Proof of Theorem~\ref{thm:LGN}]
Since $W_t:=e^{-\lambda t}Z_t(h)$ forms a non-negative martingale, it converges almost-surely to a non-negative random variable $W$. Proposition~\ref{prop:moments}, applied to $f=h$, implies that $(W_t)_{t\geq 0}$ is bounded in $L^\kappa(\mathbb P_x)$ and thus converges in $L^\kappa(\mathbb P_x)$.

Now, if  $f\in \mathcal B(V)$ with $\gamma (f)=0$ then Proposition~\ref{prop:moments} entails that $e^{-\lambda \kappa t}\moment_t^{(\kappa)}f(x)$ converges to $0$ for all $x\in \mathcal X$. Since $\kappa$ is even, this implies that $e^{-\lambda t}Z_t(f)$ converges to $0$ in $L^\kappa(\mathbb P_x).$ This ends the proof since
\[e^{-\lambda  t} Z_t(f) = e^{-\lambda t}Z_t(f-\gamma (f) h) + \gamma(f) W_t \xrightarrow[L^\kappa(\mathbb P_x)]{t\to +\infty} \gamma (f) W.\]
\end{proof}

\subsection{Rate of convergence of the martingale} In this section, we compute the speed of convergence of the martingale $(W_t)_{t\geq 0}$ both in $L^2(\mathbb P_x)$ and in $L^4(\mathbb P_x)$, quantities which will be useful in the next section. 

For this part, we will suppose that Assumptions~\ref{assu:moments} and~\ref{assu:eigenelements} for some even $\kappa$ and hold. This allows us to define the function $\psi_t^{(k)}:\mathcal X \longrightarrow \RR_+$, for $k\leq \kappa$, given by
\begin{equation}\label{eq:psi}
    \psi_t^{(k)}(x)= \mathbb E_x((W_t-h(x))^k)\xrightarrow[]{t\to+\infty}\psi_\infty^{(k)}(x)=\mathbb E_x((W-h(x))^k),
\end{equation}
where the convergence holds thanks to Theorem~\ref{thm:LGN}. 

\begin{lemma}\label{lem:speed_convergence_2}
If Assumptions~\ref{assu:moments} and~\ref{assu:eigenelements} hold for $\kappa\geq 2$, then for all $x\in \mathcal X$
\[\lim_{t\to +\infty} e^{\lambda t}\mathbb E_x((W-W_t)^2) = \gamma\left(\psi^{(2)}_\infty\right)h(x).\]
\end{lemma}

\begin{remark}\label{rem:non_degenerate}
    One has that the quantity $\gamma\left(\psi^{(2)}_\infty\right)$ is strictly positive when the process is not trivial. In fact, we note that for $x\in \mathcal X$, $\psi^{(2)}_\infty (x) = \text{Var}_x(W)$, and that, for all $t>0$,
    \[W =\sum_{u\in \alives_t} W^{u,t}, \]
    where $W^{u,t} = \lim\limits_{r\to +\infty} e^{\lambda r} Z_{t,t+r}^u(h)$. Since the law of $W^{u,t}$ conditioned to $\mathcal F_t$ is the same law of $W$, then
    \[\text{Var}_x(W) = \text{Var}_x(|\alives_t|)\mathbb E_x(W)^2 + \text{Var}_x(W)\mathbb E_x(|\alives_t|)\geq \text{Var}_x(|\alives_t|)\mathbb E_x(W)^2 = \text{Var}_x(|\alives_t|)h^2(x),\]
    where the last term is not zero when the process is not trivial.
\end{remark}
\begin{proof}[Proof of Lemma~\ref{lem:speed_convergence_2}]
Theorem~\ref{thm:LGN} implies that 
\[\mathbb E_x((W-W_t)^2) = \lim_{s\to +\infty}\mathbb E_x((W_{t+s}-W_t)^2).\]
In addition,
\begin{align*}
    \mathbb E_x((W_{t+s}-W_t)^2) &=e^{-2\lambda t} \mathbb E_x\left(\left(\sum_{u\in \alives_t}e^{-\lambda s}Z_{t,t+s}^u(h) - h(X_t^u)\right)^2\right)\\
    &=e^{-2\lambda t}\mathbb E_x\left(\sum_{u\in \alives_t}(e^{-\lambda s}Z_{t,t+s}^u(h) - h(X_t^u))^2\right),
\end{align*}
where we used the fact that $\mathbb E_x(W_s - h(x)) = 0$ for all $x\in \mathcal X$. Hence,
\begin{align*}
    \mathbb E_x((W_{t+s}-W_t)^2)&=e^{-2\lambda t} \mathbb E_x\left(\sum_{u\in \alives_t} \mathbb E_x((e^{-\lambda s}Z_{t,t+s}^u(h) - h(X_t^u))^2\mid \mathcal F_t)\right)\\
    &=e^{-2\lambda t}\mathbb E_x\left( \sum_{u\in \alives_t} \psi_s^{(2)}(X_t^u)\right)
    =e^{-2\lambda t}\moment_t\psi_s^{(2)}(x).
\end{align*}
Now by Proposition~\ref{prop:moments}, we have that the existence of a constant $C>0$ such that
\begin{equation}
\label{eq:boundedpsi2}
|\psi_s^{(2)}(x)|\leq \mathbb{E}_x(W_s^2) +  h(x)^2 \leq e^{-2\lambda s} \moment_s^{(2)} h(x) + V^2(x) \leq CV^2(x).
\end{equation}
By Assumption~\ref{assu:eigenelements}, $V^2$ is integrable over the measure $(\delta_x \moment_t)$ then the Lebesgue dominated convergence theorem gives that  
\begin{equation}\label{eq:speed_convergence_2}
    \mathbb E_x\left((W-W_t)^2\right) = e^{-2\lambda t}\moment_t\psi_\infty^{(2)}(x).
\end{equation} 

We conclude that 
\[\lim_{t\to +\infty} e^{\lambda t}\mathbb E_x((W-W_t)^2) = \lim_{t\to +\infty} e^{-\lambda t}\moment_t\psi_\infty^{(2)}(x) = \gamma\left(\psi_\infty^{(2)}\right)h(x).\]
\end{proof}
\begin{lemma}\label{lem:speed_convergence_4} 
If Assumptions~\ref{assu:moments} and~\ref{assu:eigenelements} hold for $\kappa\geq 4$, then  there exists $C>0$ such that for all $t\geq 0$ and all $x\in \mathcal X$
\[\mathbb E_x((W-W_t)^4) \leq C e^{-2\lambda t} V^4(x).\]
\end{lemma}

\begin{proof}
Theorem~\ref{thm:LGN} implies that 
\[\mathbb E_x((W-W_t)^4) = \lim_{s\to +\infty}\mathbb E_x((W_{t+s}-W_t)^4).\]
In addition, using again that $\mathbb E_x(W_s - h(x)) = 0$ for all $x\in \mathcal X$
\begin{align*}
    \mathbb E_x((W_{t+s}-W_t)^4) &=e^{-4\lambda t} \mathbb E_x\left(\left(\sum_{u\in \alives_t}e^{-\lambda s}Z_{t,t+s}^u(h) - h(X_t^u)\right)^4\right)\\
    &=e^{-4\lambda t}\mathbb E_x\left(\sum_{u\in \alives_t}(e^{-\lambda s}Z_{t,t+s}^u(h) - h(X_t^u))^4\right)\\
    &+ 6 \mathbb E_x\left(\sum_{u\in \alives_t, u\neq v}(e^{-\lambda s}Z_{t,t+s}^u(h) - h(X_t^u))^2 (e^{-\lambda s} Z_{t,t+s}^v(h) - h(X_t^v))^2\right)\\
    &= e^{-4\lambda t}\moment_t\psi_s^{(4)}(x) + 3 e^{-4\lambda t} \mathbb E_x\left(\left( \sum_{u\in \alives_t} \psi_s^{(2)}(X^u_t)\right)^2 \right) - 3 e^{-4\lambda t} \mathbb E_x\left( \sum_{u\in \alives_t} \psi_s^{(2)}(X^u_t)^2 \right)\\
    &\leq e^{-4\lambda t}\moment_t\psi_s^{(4)}(x) + 3 e^{-4\lambda t}\moment_t^{(2)} \psi_s^{(2)}(x).
\end{align*}
Using similar arguments to \eqref{eq:boundedpsi2}, one can prove that there exists a constant $C_1>0$ such that 
\[\mathbb E_x((W_{t+s}-W_t)^4)\leq C_1e^{-4\lambda t}\left(\moment_tV^4(x) + \momenttwo_t[V^2](x)\right).\]

The function $x\mapsto \moment_tV^4$ can be bounded using Assumption~\ref{assu:moments}, whilst for the second function we have that, since Assumptions~\ref{assu:moments} and~\ref{assu:eigenelements} hold for $V$ with $\kappa=4$, then they hold for $V^2$ with $\kappa=2$, and thus we can apply Proposition~\ref{prop:moments} with $V^2$, obtaining,
\[\momenttwo_t[V^2](x) \leq C_2e^{2\lambda t}V^2(x),\ \forall x\in \mathcal X,\]
for some $C_2>0$. In conlusion, there exists $C>0$ such that
\[\mathbb E_x ((W_{t+s}-W_t)^4) \leq Ce^{-2\lambda t}V^4(x),\]
and the proof follows by letting $s\to +\infty$.
\end{proof}
\section{Proof of the central limit theorems by Stein's method}\label{sec:main_proof}

\subsection{Study of the moments of the fluctuations} Before diving into the proofs of our central limit theorems, we study the properties of the moments of the fluctuations of the process. For that end, we define for $k\in \NN$, $f\in \mathcal B(V)$, $x\in \mathcal X$ and $t\geq 0$ the quantity
\begin{equation}\label{eq:goticF}
\mathfrak F_t^{(k)}[f](x) := \mathbb E_x\left(\left|Z_t(f) - e^{\lambda t}\gamma (f)W\right|^k\right).
\end{equation}

Our goal for this part is to study the limit when $t\to +\infty$ of $\mathfrak F_t^{(k)}[f](x)$, properly rescaled depending whether we are in the small or critical branching case, for $k=2$ and $k=3$. For $f\in \mathcal B(V)$, by setting $\widehat f = f - \gamma(f) h$, we have
\begin{align}
    \mathfrak F_t^{(2)}[f](x) & = \mathbb E_x\left(\left(Z_t(f)-e^{\lambda t}\gamma (f) W\right)^2\right) \nonumber\\
    & = \mathbb E_x\left(\left(Z_t(\widehat f) + e^{\lambda t}(W_t-W)\right)^2\right) \nonumber\\
    & = \mathbb E_x \left(Z_t(\widehat f)^2 \right) + 2e^{\lambda t}\gamma (f)\mathbb E_x\left(Z_t(\widehat f)(W_t-W)\right) + \gamma(f)^2 e^{2\lambda t} \mathbb E_x\left((W_t-W)^2\right) \nonumber\\
    &= \momenttwo_t[\widehat f](x) + \gamma(f)^2e^{2\lambda t} \mathbb E_x \left((W_t-W)^2\right)\nonumber\\
    &= \momenttwo_t[\widehat f](x) + \gamma (f)^2 \moment_t\psi_\infty^{(2)}(x),\label{eq:fluctuation_decomp}
\end{align}
where in the last equality we used \eqref{eq:speed_convergence_2}.
The second term can be dealt with by using Lemma~\ref{lem:speed_convergence_2}. Our next result deals with the first term.

\begin{lemma}\label{lem:kappa}
If Assumptions~\ref{assu:moments} and~\ref{assu:eigenelements} hold for $\kappa\geq 2$, then for all $f\in \mathcal B(V)$ with $\gamma(f)=0$ there exists $\eta_s(f)\geq 0$ such that in the small branching case ($\lambda < 2\rho$) for all $x\in \mathcal X$,
\begin{equation}\label{eq:eta_s}
    \lim_{t\to +\infty} e^{-\lambda t} \momenttwo_t [f](x) =  \eta_s(f)h(x).
\end{equation}
Similarly, in the critical branching case ($\lambda = 2\rho$) there exists $\eta_c(f)\geq 0$ such that for all $x\in \mathcal X$,
\begin{equation}\label{eq:eta_c}
    \lim_{t\to +\infty} \frac{e^{-\lambda t}}{t} \momenttwo_t [f](x)=\eta_c(f)h(x).
\end{equation}
\end{lemma}

\begin{proof} Consider $f\in \mathcal B(V)$ with $\gamma(f)=0$ and suppose without loss of generality that $\|f\|_V=1$. Using the branching property we can show that for all $r,s >0$ and all $x\in\mathcal X$,
\begin{equation}\label{eq:decompo-secondmoment2}
    \momenttwo_{r+s}[f](x) = \moment_r\momenttwo_s[f](x)+\Psi_{r,s}[f](x),
\end{equation}
with
\[\Psi_{r,s}[f](x) = \mathbb E_x\left(\sum_{\substack{u,v\in V_r\\ u\neq v}}\moment_sf(X_r^u)\moment_sf(X_r^u)\right).\]

We start by noticing that by applying Assumption~\ref{assu:eigenelements} and since $\gamma(f)=0$,
\begin{equation}\label{eq:Psi_bound}
|\Psi_{r,s}[f](x)| \leq e^{2(\lambda - \rho)s}\momenttwo_r[V](x).
\end{equation}
Then, if in the small branching case the sequence $\left(e^{-\lambda t}\momenttwo_t[f](x)\right)_{t\geq 0}$ converges to some limit $g_f(x)$, then by \eqref{eq:decompo-secondmoment2} we have for all $r>0$
\begin{align*}
    g_f(x) &= \lim_{s\to +\infty} e^{-\lambda (r+s)} \momenttwo_{r+s}[f](x)\\
    &=\lim_{s\to +\infty} \left(e^{-\lambda r} \moment_r \left(e^{-\lambda s} \momenttwo_s[f]\right) (x) + e^{-\lambda (r+s)}\Psi_{r,s}[f](x)\right)\\
    &=e^{-\lambda r}\moment_r g_f(x),
\end{align*}
where we used the Dominated Convergence Theorem and that the function $s\mapsto e^{-\lambda s}\momenttwo_s[f](x)$ is bounded in $\mathcal B(V)$ by Proposition~\ref{prop:moments}. This means that $g_f\in \mathcal B(V)$ is a positive eigenfunction of $\moment$, and thus by Assumption~\ref{assu:eigenelements} there exists $\eta_s(f)\geq 0$ such that $g(x) = \eta_s(f)h(x)$. The reasoning for the critical branching case is analogous, so the rest of the proof is dedicated to show that the two limits in \eqref{eq:eta_s} and \eqref{eq:eta_c} exist. 

For the rest of the proof we fix $t_0\in (0,1)$. An induction argument using \eqref{eq:decompo-secondmoment2} shows that for all $n\in \NN$,
\begin{align*}
    \momenttwo_{(n+1)t_0}[f](x) = \moment_{nt_0}\momenttwo_{t_0}[f](x)+\sum_{k=0}^{n-1} \moment_{kt_0} \Psi_{t_0,(n-k)t_0}[f](x).
\end{align*}
Hence, if we define for $t>2t_0$ the functions
\[n(t):=\sup\{m\in \NN\ :\ (m+1)t_0<t\},\text{ and }\delta(t):= t-(n(t)+1)t_0\in (0,t_0],\]
where $t>2t_0$ ensures that $n(t)\neq 0$. Then, for all $t>2t_0$, 
\begin{align*}
    \momenttwo_t[f](x) & = \momenttwo_{(n(t)+1)t_0+\delta(t)}[f](x)\\
    & = \moment_{\delta(t)}\momenttwo_{(n(t)+1)t_0}[f](x) + \Psi_{\delta(t),(n(t)+1)t_0}[f](x)\\
    & = \moment_{t-t_0}\momenttwo_{t_0}[f](x) + \sum_{k=0}^{n(t)-1}\moment_{kt_0+\delta(t)}\Psi_{t_0,(n(t)-k)t_0}[f](x)+\Psi_{\delta(t),(n(t)+1)t_0}[f](x).
\end{align*}
Our goal is then to prove that all three terms converge, after proper rescaling, when $t\to +\infty$. We start by pointing out that for all $r,s>0$,
\begin{equation*}\label{eq:Psi_expansion}
    \Psi_{r,s}[f](x)=\momenttwo_r\left[\moment_sf\right](x)-\moment_r\left(\moment_sf\right)^2(x),
\end{equation*}
and thus, the sum in the middle can be decomposed as
\begin{align*}\label{eq:sum_decomposition}
    \sum_{k=0}^{n(t)-1}\moment_{kt_0+\delta(t)}\Psi_{t_0,(n(t)-k)t_0}[f](x)&=\sum_{k=0}^{n(t)-1}\moment_{kt_0+\delta(t)}\momenttwo_{t_0}\left[\moment_{(n(t)-k)t_0}f\right](x)\\ & - \sum_{k=0}^{n(t)-1}\moment_{kt_0+\delta(t)}(\moment_{(n(t)-k)t_0}f)^2(x),
\end{align*}
where both sums on the right-hand side are formed of non-negative terms. Then, we only need to provide upper bounds for both sums. For the first one, since $\gamma(\moment_{(n(t)-k)t_0}f)=0$, by applying Proposition~\ref{prop:moments} both in the small and critical branching case (we recall that $t_0\in (0,1)$) and then Assumption~\ref{assu:eigenelements} with $\gamma(f)=0$, there exists a constant $C>0$, such that 
\begin{align*}
\momenttwo_{t_0}\left[\moment_{(n(t)-k)t_0}f\right](x)\leq Ce^{2(\lambda - \rho)(n(t)-k)t_0}V(x),
\end{align*}
Hence, again by Proposition~\ref{prop:moments}, there exists a constant $K_1>0$ such that
\begin{align*}
    \sum_{k=0}^{n(t)-1}\moment_{kt_0+\delta(t)}\momenttwo_{t_0}\left[\moment_{(n(t)-k)t_0}f\right](x)&\leq C\sum_{k=0}^{n(t)-1}e^{2(\lambda-\rho)(n(t)-k)t_0}\moment_{kt_0+\delta(t)}V(x)\\
    &\leq K_1e^{2(\lambda -\rho)n(t)t_0}\sum_{k=0}^{n(t)-1}e^{-2(\lambda-\rho)kt_0}e^{\lambda (kt_0+\delta(t))}V(x)\\
    &\leq K_1 e^{\lambda t_0}e^{2(\lambda - \rho)t}V(x)\sum_{k=0}^{n(t)-1} e^{-(\lambda - 2\rho)kt_0}.
\end{align*}
In the small branching case, we have then that there exists $\widehat K_1>0$ such that 
\begin{align*}
    e^{-\lambda t}\sum_{k=0}^{n(t)-1}\moment_{kt_0+\delta(t)}\momenttwo_{t_0}[\moment_{(n(t)-k)t_0}f](x)& \leq \widehat K_1 e^{(\lambda -2\rho)t}\left(e^{-(\lambda -2\rho)n(t)t_0}-1\right)V(x)\\
    &\leq \widehat K_1 \left(e^{(\lambda - 2\rho)(t_0+\delta(t))} - e^{(\lambda - 2\rho)t}\right)V(x)\\
    &\leq \widehat K_1 \left(e^{(\lambda - 2\rho)t_0}-1\right)V(x)\\
    &<+\infty,
\end{align*}
where we used that in the small branching case $\lambda-2\rho<0$. For the critical branching case where $\lambda=2\rho$ we have
\begin{align*}
\frac{e^{-\lambda t}}{t} \sum_{k=0}^{n(t)-1}\moment_{kt_0+\delta(t)}\momenttwo_{t_0}\left[\moment_{(n(t)-k)t_0}f\right](x) &\leq K_1 e^{\lambda t_0}V(x)\frac{n(t)}{t}\\
&\leq K_1\frac{e^{\lambda t_0}}{t_0}V(x)\\
&<+\infty.
\end{align*}
We now turn our attention to the second sum. Assumption~\ref{assu:eigenelements} with $\gamma(f)=0$ implies that 
\[\left(\moment_{(n(t)-k)t_0}f(x)\right)^2\leq e^{2(\lambda-\rho)(n(t)-k)t_0}V^2(x).\]
Then, by Assumption~\ref{assu:moments}, there exists a constant $K_2>0$ such that
\[\sum_{k=0}^{n(t)-1} \moment_{kt_0+\delta(t)} \left(\moment_{(n(t)-k)t_0}f\right)^2(x)\leq K_2\sum_{k=0}^{n(t)-1} e^{2(\lambda - \rho)(n(t)-k)t_0}\moment_{kt_0+\delta(t)}V^2(x),\]
which is the same kind of bound that we found for the first sum but with $V^2$ instead of $V$, and so the proof follows similarly.

To conclude the proof, in the small branching case we have that 
\[e^{-\lambda t}\momenttwo_t[f](x) = e^{-\lambda t}\moment_{t-t_0}\momenttwo_{t_0}[f](x) + e^{-\lambda t} \sum_{k=0}^{n(t)-1}\moment_{kt_0+\delta(t)} \Psi_{t_0,(n(t)-k)t_0}[f](x) + e^{-\lambda t}\Psi_{\delta(t),(n(t)+1)t_0}[f](x).\]
For the first term, Assumption~\ref{assu:eigenelements} implies that
\[e^{-\lambda t}\moment_{t-t_0}\momenttwo_{t_0}[f](x)\xrightarrow[]{t\to +\infty}e^{\lambda t_0} \gamma\left(\momenttwo_{t_0}[f]\right),\]
we have already shown that the sum is convergent, and for the last term we have using \eqref{eq:Psi_bound} and Proposition~\ref{prop:moments} that
\[\left|e^{-\lambda t}\Psi_{\delta(t),(n(t)+1)t_0}[f](x)\right|\leq e^{-\lambda t}e^{2(\lambda - \rho)(n(t)+1)t_0} \momenttwo_{\delta(t)} [V](x) \leq K_3 e^{(\lambda - 2\rho)t}e^{2\lambda t_0}V^2(x),\]
where the right-hand side tends to zero when $t$ goes to infinity. The reasoning in the critical branching case is similar, and this concludes the proof.
\end{proof}

Combining the previous lemma with the decomposition for $\mathfrak F_t^{(2)}$ provided in~\eqref{eq:fluctuation_decomp} and Lemma~\ref{lem:speed_convergence_2}, the following direct corollary allows us to define the quantities $\sigma_{f,s}^2$ and $\sigma_{f,c}^2$ appearing in Theorems~\ref{thm:CLT_SBC} and~\ref{thm:CLT_CBC}.

\begin{corollary}\label{cor:fluctuations_order2}
If Assumption~\ref{assu:moments} for $\kappa\geq 2$ and Assumption~\ref{assu:eigenelements} hold, then for all $f\in \mathcal B(V)$ and all $x\in \mathcal X$ we have that in the small branching case
\begin{equation*}
    e^{-\lambda t}\mathfrak F^{(2)}_t[f](x)\xrightarrow[]{t\to +\infty}\underbrace{\left(\eta_s(\widehat{f})+\gamma(f)^2\gamma\left(\psi_\infty^{(2)}\right)\right)}_{=:\sigma_{f,s}^2}h(x),
\end{equation*}
and in the critical branching case
\begin{equation*}
    \frac{e^{-\lambda t}}{t}\mathfrak F^{(2)}_t[f](x)\xrightarrow[]{t\to +\infty} \underbrace{\eta_c(\widehat f)}_{=:\sigma_{f,c}^2}h(x),
\end{equation*}
where $\widehat f= f-\gamma(f)h$, $\eta_s$ and $\eta_c$ are given by Lemma~\ref{lem:kappa} and $\psi_\infty^{(2)}$ is defined in~\eqref{eq:psi}. 

\end{corollary}

\begin{remark}
    Since $\widehat h \equiv 0$, we have that for Theorem~\ref{thm:CLT_Martingale}, $\sigma_h^2$ will be equal to $\gamma \left(\psi_\infty^{(2)}\right)$. 
\end{remark}
In the following lemma we provide an explicit speed for the convergence of the previous lemma in the small branching case. However, our mild assumptions do not allow us to extend this result to the critical branching case.

\begin{lemma}\label{lem:speed_fluctuations2}
IfAssumption~\ref{assu:moments} for $\kappa\geq 2$ and Assumption~\ref{assu:eigenelements} hold, then in the small branching case we have that there exists $K>0$ such that for all $f\in \mathcal B(V)$ and all $t>0$,
\[\left\|e^{-\lambda t}\mathfrak F_t^{(2)}[f]-\sigma^2_{f,s}h\right\|_{V^2} \leq K\left(\|\widehat f\|_V^2+\gamma(f)^2\|\psi_\infty^{(2)}\|_{V^2}\right)e^{(\nf{\lambda}{2}-\rho)t},\]
where $\widehat f = f -\gamma(h)f$ and $\sigma_{f,s}^2 = \eta_s(\widehat f)+\gamma(f)^2\gamma\left(\psi_\infty^{(2)}\right)$.
\end{lemma}

\begin{proof}
    Consider $f\in \mathcal B(V)$ and set $\widehat f = f-\gamma(f)h$. Decomposing $\mathfrak F_t^{(2)}[f]$ as in \eqref{eq:fluctuation_decomp}, the definition of $\sigma_{f,s}^2$ and  Lemma~\ref{lem:speed_convergence_2}, we have that for all $t>0$,
    \begin{align*}
        \left\|e^{-\lambda t}\mathfrak F_t^{(2)}[f] -\sigma^2_{f,s}h\right\|_{V^2} \leq \left\|e^{-\lambda t} \momenttwo_t[\widehat f] - \eta_s(\widehat f)h\right\|_{V^2} + \gamma(f)^2\left\|e^{-\lambda t}\moment_t\psi_\infty^{(2)}-\gamma\left(\psi_\infty^{(2)}\right)h\right\|_{V^2}.
    \end{align*}
    For the second term on the right-hand side, since $\psi_\infty^{(2)}\in \mathcal B(V^2)$, we apply Assumption~\ref{assu:eigenelements} which entails that
    \[\left\|e^{-\lambda t}\moment_t\psi_\infty^{(2)}-\gamma\left(\psi_\infty^{(2)}\right)h\right\|_{V^2}\leq a_2e^{-\rho t}\|\psi_\infty^{(2)}\|_{V^2}\leq a_2 e^{(\nf{\lambda}{2}-\rho)t}\|\psi_\infty^{(2)}\|_{V^2},\]
    whilst for the first term we will prove the following more general result: there exists $C>0$ such that for all $g\in \mathcal B(V)$ with $\gamma(g)=0$ and all $t_1,t_2>0$ 
    \begin{equation}\label{eq:lem_speed_fluctuations2}
        \left\|e^{-\lambda (t_1+t_2)}\momenttwo_{t_1+t_2}[g] - \eta_s(g)h\right\|_{V^2}\leq C\|g\|_V^2\left(e^{-\rho t_1} + e^{(\lambda - 2\rho) t_2}\right).
    \end{equation}
    This is enough to conclude since, given that $\gamma (\widehat f) = 0$, we then have 
    \begin{align*}
        \left\|e^{-\lambda t}\momenttwo_t[\widehat f] - \eta_s(\widehat f) h\right\|_V &\leq C \|\widehat f\|_V^2\inf_{\substack{t_1,t_2>0\\ t_1+t_2=t}} \left(e^{-\rho t_1}+e^{(\lambda - 2\rho)t_2}\right)\\ &= C \|\widehat f\|_V^2 \inf_{r\in [0,t]} \left(e^{-\rho(t-r)}+e^{(\lambda -2\rho)r}\right)\\
        &= C\|\widehat f\|_V^2e^{(\nf{\lambda}{2}-\rho) t},
    \end{align*}
    and the result of the lemma follows from here. The rest of the proof is dedicated to prove \eqref{eq:lem_speed_fluctuations2}. 
    
    For all $t_1,t_2>0$, by using the branching property at time $t_1$ we have that for all $x\in \mathcal X$
    \begin{equation}\label{eq:branching_decomposition}
        \momenttwo_{t_1+t_2}[g](x) =  \moment_{t_1}\momenttwo_{t_2}[g](x) + \momenttwo_{t_1}\left[\moment_{t_2}g\right](x) - \moment_{t_1}(\moment_{t_2}g)^2(x).
    \end{equation}
    In particular, integrating over $\gamma$ and scaling by $e^{-\lambda (t_1+t_2)}$, we have that for all $t_1,t_2>0$
    \begin{align}\label{eq:lem_fluctuations2_cauchy}
    \left| e^{-\lambda (t_1+t_2)}\gamma \left(\moment_{t_1+t_2}^{(2)} [g]\right)  - e^{-\lambda t_2}  \gamma \left(\moment_{t_2}^{(2)} [g]\right) \right| \leq e^{-\lambda (t_1+t_2) } \gamma \left(\moment_{t_1}^{(2)} \left[\moment_{t_2} g\right]\right) + e^{-\lambda t_2}\gamma \left( (\moment_{t_2} g)^2\right).
    \end{align}
    Since $\gamma(\moment_{t_2}g) = 0$, Proposition~\ref{prop:moments} for the small branching case entails that there exists $C_1>0$ such that
    \[e^{-\lambda (t_1+t_2)} \momenttwo_{t_1}[\moment_{t_2}g](x) \leq C_1\|\moment_{t_2}g\|_V^2 e^{-\lambda t_2}V^2(x).\]
    In addition, Assumption~\ref{assu:eigenelements} with $\gamma(g)=0$ gives
    \begin{equation}\label{eq:moment1_gamma0}
    |\moment_{t_2}g(x)|\leq a_2 e^{(\lambda -\rho)t_2}\|g\|_VV(x).
    \end{equation}
    Hence, 
    \begin{equation}\label{eq:useful_fluctuations}
        e^{-\lambda (t_1+t_2)}\momenttwo_{t_1}\left[\moment_{t_2}g\right](x) \leq C_1a_2^2\|g\|_V^2e^{(\lambda -2\rho)t_2} V^2(x),
    \end{equation}
   and thus by integrating with respect to $\gamma$,
   \[e^{-\lambda (t_1+t_2)}\gamma \left(\momenttwo_{t_1}[\moment_{t_2}g]\right) \leq C_1a_2^2\|g\|_V^2 \gamma(V^2) e^{(\lambda -2\rho)t_2}.\]
   Similarly, using \eqref{eq:moment1_gamma0} again,
   \[e^{-\lambda t_2}\gamma\left((\moment_{t_2}g)^2\right) \leq a_2^2 \|g\|_V^2\gamma(V^2)e^{(\lambda - 2\rho)t_2}.\]
   These bounds in \eqref{eq:lem_fluctuations2_cauchy} imply by letting $t_1$ go to infinity (and since $\gamma(h)=1$) that there exists $C_2>0$ such that 
   \begin{equation}\label{eq:lem_fluctuations_conclusion1}
        \left|e^{-\lambda t_2}\gamma\left(\momenttwo_{t_2}[g]\right) -\eta_s(g)\right| \leq C_2\|g\|_V^2e^{(\lambda -2\rho)t_2}.
   \end{equation}
    On the other hand, using again~\eqref{eq:branching_decomposition}, we have that for all $t_1,t_2>0$ and all $x\in \mathcal X$,
    \begin{align*}
        &\left|e^{-\lambda (t_1+t_2)}\momenttwo_{t_1+t_2}[g](x)-e^{-\lambda t_2}\gamma\left(\momenttwo_{t_2}[g]\right)h(x)\right|\\
        &\quad \leq e^{-\lambda(t_1+t_2)}\momenttwo_{t_1}[\moment_{t_2}g](x) \\ &\qquad +e^{-\lambda(t_1+t_2)}\moment_{t_1}\left(\moment_{t_2}g\right)^2(x) +\left|e^{-\lambda t_1}\moment_{t_1}\left(e^{-\lambda t_2}\momenttwo_{t_2}[g]\right)(x) -\gamma \left( e^{-\lambda t_2}\momenttwo_{t_2}(g)\right)h(x)\right|.
    \end{align*}
    The first term on the right-hand side is bounded in~\eqref{eq:useful_fluctuations}. For the second term we can use Assumption~\ref{assu:eigenelements} and then~\eqref{eq:moment1_gamma0} which implies the existence of a constant $C_3>0$ such that 
    \[e^{-\lambda(t_1+t_2)}\moment_{t_1}\left(\moment_{t_2}g\right)^2(x)\leq C_3 e^{(\lambda -2\rho)t_2}\|g\|_V^2 V^2(x).\]
    For the third term, Assumption~\ref{assu:eigenelements} and Proposition~\ref{prop:moments} yield
    \begin{align*}
        \left|e^{-\lambda t_1}\moment_{t_1}\left(e^{-\lambda t_2}\momenttwo_{t_2}[g]\right)(x) -\gamma \left( e^{-\lambda t_2}\momenttwo_{t_2}(g)\right)h(x)\right|&\leq a_2 e^{-\rho t_1} \left\|e^{-\lambda t_2}\momenttwo_{t_2}[g]\right\|_{V^2} V^2(x)\\
        &\leq C_4 \|g\|^2_V e^{-\rho t_1}V(x),
    \end{align*}
    for some $C_4>0$. Thus, there exists $C_5>0$ such that
    \begin{equation}\label{eq:lem_fluctuations_conclusion2}
        \left\|e^{-\lambda (t_1+t_2)}\momenttwo_{t_1+t_2}[g]-e^{-\lambda t_2}\gamma\left(\momenttwo_{t_2}(g)\right)h\right\|_{V^2} \leq C_5\|g\|_V^2 \left(e^{-\rho t_1} + e^{(\lambda -2\rho) t_2}\right).
    \end{equation}
    Combining~\eqref{eq:lem_fluctuations_conclusion1} and~\eqref{eq:lem_fluctuations_conclusion2}, we conclude that there exists $C>0$ such that
    \begin{align*}
        &\left\|e^{-\lambda (t_1+t_2)}\momenttwo_{t_1+t_2}[g] - \eta_s(g)h\right\|_{V^2} \\
        &\qquad \leq \left\|e^{-\lambda (t_1+t_2)}\momenttwo_{t_1+t_2}[g]-e^{-\lambda t_2}\gamma\left(\momenttwo_{t_2}(g)\right)h\right\|_{V^2} + \left|e^{-\lambda t_2}\gamma\left(\momenttwo_{t_2}[g]\right) -\eta_s(g)\right|\\
        &\qquad \leq C\|g\|_V^2 \left(e^{-\rho t_1} + e^{(\lambda -2\rho)t_2}\right),
    \end{align*}
    and the proof complete.
\end{proof}

We finish this section with a result on the moments of third order of the fluctuations of the process.

\begin{lemma}\label{lem:fluctuations_order3}
If Assumptions~\ref{assu:moments} and~\ref{assu:eigenelements} hold for $\kappa\geq 4$, then for all $f\in \mathcal B(V)$ there exists a constant $C>0$ such that for all $x\in \mathcal X$ and all $t\geq 1$, in the small branching case $(\lambda < 2\rho)$, 
\[e^{-\nf{-3\lambda}{2}}\mathfrak F_t^{(3)}[f](x)\leq CV^3(x),\ \forall x\in \mathcal X,\]
whilst in the critical branching case $(\lambda = 2\rho)$
\[\frac{e^{\nf{-3\lambda t}{2}}}{t^{-\nf{3}{2}}}\mathfrak F_t^{(3)}[f](x) \leq CV^3(x), \forall x\in \mathcal X.\]
\end{lemma}

\begin{proof}
Consider $f\in \mathcal B(V)$ and set $\widehat f = f-\gamma(f)h$. We have that for all $t\geq 0$ and $x\in \mathcal X$
\begin{align*}
    \mathfrak F_t^{(3)}[f](x)&=\mathbb E_x\left(\left|Z_t(f) - e^{\lambda t}\gamma (f) W\right|^3\right)\\
    &= \mathbb E_x\left(\left|Z_t(\widehat f) + \gamma(f)e^{\lambda t}(W_t-W)\right|^3\right)\\
    &\leq \mathbb E_x\left(|Z_t(\widehat f)|^3\right)+ 3\gamma(f)e^{\lambda t}\mathbb E_x\left(|Z_t(\widehat f)|^2|W_s-W|\right)\\
    &\qquad +3\gamma(f)^2e^{2\lambda t} \mathbb E_x\left(|Z_t(\widehat f)|(W_t-W)^2\right)+\gamma(f)^3e^{3\lambda t}\mathbb E_x\left(|W_s-W|^3\right)
\end{align*}
We proceed to bound the four terms separately. We deal only with the small branching case, the other case being analogous. 

In what follows $C>0$ represents a constant that depends on $f$ and that can vary from line to line. By the Cauchy-Schwarz inequality and Proposition~\ref{prop:moments} for the small branching case, since $\gamma(\widehat f)=0$, we have for the first term,
\begin{align*}
    \mathbb E_x\left(|Z_t(\widehat f)|^3\right) &\leq \mathbb E_x\left(Z_t(\widehat f)^2\right)^{\nf{1}{2}}\mathbb E_x\left(Z_t(\widehat f)^4\right)^{\nf{1}{2}}\\
    &= \left(\moment_t^{(2)}[\widehat f](x)\right)^{\nf{1}{2}}\left(\moment_t^{(4)}[\widehat f](x)\right)^{\nf{1}{2}}\\
    &\leq C\left(\|\widehat f\|_V^2 e^{\lambda t} V^2(x)\right)^{\nf{1}{2}} \left( \|\widehat f\|_V^2 e^{2\lambda t}V^4(x)\right)^{\nf{1}{2}}\\
    &\leq Ce^{\nf{3\lambda t}{2}}V^3(x).
\end{align*}
For the second term we use in addition Lemma~\ref{lem:speed_convergence_2} obtaining,
\begin{align*}
    3\gamma(f)e^{\lambda t}\mathbb E_x\left(|Z_t(\widehat f)|^2|W_s-W|\right)&\leq C\gamma(f)e^{\lambda t} \left(\moment_t^{(4)}\widehat f(x)\right)^{\nf{1}{2}}\mathbb E_x\left((W-W_t)^2\right)^{\nf{1}{2}}\\
    &\leq C 3\gamma(f) e^{\lambda t} \left(\|\widehat f\|_V^4 e^{2\lambda t}V^4(x)\right)^{\nf{1}{2}}\left(e^{-\lambda t}\|\psi_\infty^{(2)}\|_{V^2}V^2(x)\right)^{\nf{1}{2}}\\
    &\leq C e^{\nf{3\lambda t}{2}} V^3(x).
\end{align*}
The last two terms are bounded similarly and using also Lemma~\ref{lem:speed_convergence_4}.
\end{proof}


\subsection{Basics on Stein's method}\label{sec:Stein} In this section we provide a brief overview on Stein's method, which is a technique that allows to estimate the distance between a given probability distribution and a Gaussian random variable. This method can be extended to approximate any other probability law (see \cite{chen2010normal,ross2011fundamentals} for details), however we will concentrate only on the Gaussian case.

The main idea behind Stein's method is that a random variable $\mathcal Z_\sigma$ follows a Gaussian distribution of zero mean and variance $\sigma^2$ if and only if
\[\sigma^2 G'(\mathcal Z_\sigma) - \mathcal Z_\sigma G(\mathcal Z_\sigma)=0,\]
for all function $G$ regular enough. With this in mind, we set Stein's equation as
\begin{equation}\label{eq:Stein_equation}
    F(x) - \pi_\sigma(F) = \sigma^2 G'(x) - xG(x),\ \forall x\in \RR,
\end{equation}
where $\pi_\sigma$ is the law of $\mathcal Z_\sigma$. Thus, given two family of functions $\mathcal F$ and $\mathcal G$, if for all $F\in \mathcal F$ there exists $G\in \mathcal G$ such that \eqref{eq:Stein_equation} holds, then for any random variable $X$
\[d_{\mathcal F}(X,\mathcal Z_\sigma) := \sup_{F\in \mathcal F} |\mathbb E(F(X)) - \mathbb E(f(\mathcal Z_\sigma)) |\leq \sup_{G\in \mathcal G} |\mathbb E(\sigma^2 G'(X) - XG(X))|.\]

We consider now the distance $\textbf{d}$ defined in the Introduction, which is the distance $d_{\mathcal F}$ when considering
\begin{equation}\label{eq:def_F}
\mathcal F := \{F:\RR\longrightarrow \RR, \|F\|_\infty, \|F''\|_\infty, \|F''\|_\infty, \|F^{(3)}\|_\infty \leq 1\}.
\end{equation}
For this family of functions, we provide in the following proposition the corresponding family of solutions for Stein's equation.

\begin{proposition}\label{prop:Stein_solution} For any $\sigma\geq 0$ and for all $F\in \mathcal F$ there exists $G\in \mathcal G$ which is solution to \eqref{eq:Stein_equation}, where $\mathcal G$ is given by
\begin{equation*}\label{eq:set_G}
    \mathcal G:=\left\{G:\RR\longrightarrow \RR\ :\ \|G\|_\infty \leq 1,\ \|G'\|_\infty \leq \frac{1}{2},\ \|G''\|_\infty \leq \frac{1}{3} \right\}.
\end{equation*}
\end{proposition}

\begin{proof}
    The proof is very standard and we add it here for sake of completeness (see \cite{cloez2019intertwinings} for more details). We consider the alternative problem of finding a function $H:\RR\longrightarrow \RR$ such that 
    \begin{equation}\label{eq:Stein_alternative}
        F(x) - \pi_\sigma (F) = \sigma^2 H''(x) - xH'(x).
    \end{equation}
    We remark that such $H$ holds that $G:=H'$ is solution to \eqref{eq:Stein_equation}, and thus we need to prove the existence of a solution $H$ to \eqref{eq:Stein_alternative} such that $H'\in \mathcal G$. We start by defining the operator $\mathcal L$ given by 
    \[\mathcal Lh(x) = \sigma^2 h''(x) - xh'(x),\]
    and so we can rewrite \eqref{eq:Stein_alternative} as $F-\pi_\sigma(F) = \mathcal LH$. The operator $\mathcal L$ corresponds to the infinitesimal generator of the semigroup $(P_t)_{t\geq 0}$ acting on all bounded function $F:\RR\longrightarrow \RR$ by
    \[P_tF(x) = \int\limits_{-\infty}^{+\infty} F(e^{-t}x + (1-e^{-2t})y)\pi_\sigma(dy).\]
    
    For $F\in \mathcal F$, a solution to \eqref{eq:Stein_alternative} is then given by the function $H:\RR \longrightarrow\RR $ defined by 
    \[H(x) := -\int\limits_0^{+\infty} P_tF(x)dt.\]
    Indeed, by the Dominated Convergence Theorem and Kolmogorov's backward equation
    \begin{align*}
       \mathcal Lh(x) =-\int\limits_0^{+\infty} \mathcal LP_tF(x)dt= -\int\limits_0^{+\infty} \dfrac{d}{dt}P_tF(x)dt= F(x) - \lim_{t\to +\infty} P_tF(x)= F(x) - \pi_\sigma(F),
    \end{align*}
    where for the limit we used that since $F\in \mathcal F$, it is continuous.
    
    Let us now verify that $H'\in \mathcal G$. Again by the Dominated Convergence Theorem we have that 
    \[\|H'\|_\infty \leq \left|\int\limits_0^{+\infty} \dfrac{d}{dx}P_tF(x)dt \right| = \int\limits_0^{+\infty} e^{-t}P_t|F'|(x)dt \leq \|F'\|_\infty \int_0^{+\infty} e^{-t}dt \leq \|F'\|_\infty\leq 1.\]
    Similarly, it follows that $\|h''\|_\infty\leq \frac{1}{2}$ and $\|h'''\|_\infty \leq \frac{1}{3}$. The proof is then complete.
\end{proof}

We are now in position to prove Theorems~\ref{thm:CLT_Martingale}, \ref{thm:CLT_SBC} and~\ref{thm:CLT_CBC}.

\subsection{Proof of the central limit theorems}

We will start with some preliminary computations that work both in all the proofs and we will separate the cases in the end. In order to do this, we define $\mathcal S:\RR_+\longrightarrow\RR_+$ a scale function, with $\mathcal S(t) = e^{\nf{-\lambda t}{2}}$ when $f=h$ and in the small branching case, and $\mathcal S(t)=t^{\nf{-1}{2}}e^{\nf{-\lambda t}{2}}$ in the critical branching case. We also set $\sigma_f^2\geq 0$ as $\sigma_f^2=\sigma_h^2:=\gamma \left(\psi_\infty^{(2)}\right)$ when $f=h$, $\sigma_f^2=\sigma_{f,s}^2=\eta_s\left(f-\gamma(f)h\right)+\gamma(f)^2\gamma\left(\psi_\infty^{(2)}\right)$ in the small branching case and $\sigma_f^2=\sigma_{f,c}^2=\eta_c\left(f-\gamma(f)h\right)$ in the critical branching case. We recall that $\psi_\infty^{(2)}$ is defined in~\eqref{eq:psi} and that $\eta_s$ and $\eta_c$ are defined in Lemma~\ref{lem:kappa}.  

For $f\in \mathcal B(V)$ we introduce the process
\[Y_t:=\mathcal S(t)\left(Z_t(f) - e^{\lambda t}\gamma (f) W\right).\]

Consider $\mathcal Z\sim \mathcal N(0,1)$, independent of $(W_t)_{t\geq 0}$ and $F\in \mathcal F$, with $\mathcal F$ the family of functions defined in \eqref{eq:def_F}. By Proposition~\ref{prop:Stein_solution}, for each $t>0$, there exists a (random) solution $G_t$ to \eqref{eq:Stein_equation} with $\sigma^2=\sigma^2_fW_t$, which is $\mathcal F_t$-measurable (with $(\mathcal F_t)_{t\geq 0}$ the natural filtration of the process) and such that $G_t\in \mathcal G$ almost surely, with 
\[ \mathcal G:=\left\{G:\RR\longrightarrow \RR\ :\ \|G\|_\infty \leq 1,\ \|G'\|_\infty \leq \frac{1}{2},\ \|G''\|_\infty \leq \frac{1}{3} \right\}.\]
In other words, for $t,r>0$ we have
\begin{equation}\label{eq:Stein_comparison}
\left|\mathbb E_x(F(Y_{t+r}))- \mathbb E_x\left(F\left(\sigma_f\sqrt{W_t}\mathcal Z\right)\right)\right| = \left|\mathbb E_x \left(Y_{t+r}G_t(Y_{t+r}) - \sigma_f^2 W_t G_t'(Y_{t+r})\right)\right|.
\end{equation}
We proceed now to bound the right-hand side.

For $t>0$ and for each $u\in \alives_t$ we set $W^{u,t}:=\lim\limits_{\ell\to +\infty}e^{\lambda \ell}Z^u_{t,t+\ell}(h)$, then the branching property entails that for all $t,r>0$
\[Y_{t+r}=\mathcal S(t+r) \sum_{u\in \alives_t} A_{t,t+r}^u = \mathcal S(t+r)\sum_{u\in \alives_t}\underbrace{\left(Z_{t,t+r}^u(f) - e^{\lambda r}\gamma(f)W^{u,t}\right)}_{=:A_{t,t+r}^u}.\]

For $u\in \alives_t$ we also define
\[Y_{t,r}^{(-u)} = \mathcal S(t+r)\sum_{\substack{v\in \alives_t\\ v\neq u}} A_{t,t+r}^v = Y_{t+r}-\mathcal S(t+r) A_{t,t+r}^u,\]
and we remark that the branching property implies that $Y_{t,r}^{(-u)}$ is independent of $A_{t,t+r}^u$ conditionally on $\mathcal F_t$.

We have that for all $x\in \mathcal X$,
\begin{align}
    &\quad \left|\mathbb E_x\left(Y_{t+r}G_t(Y_{t+r})-\sigma_f^2 W_tG'_t(Y_{t+r})\right)\right| \label{eq:TCL_bound}\\
    &\leq \mathbb E_x\left(\left|\mathbb E_x(Y_{t+r}G_t(Y_{t+r})\mid \mathcal F_t)-\mathcal S(t+r)^2\sum_{u\in \alives_t}\mathbb E_x\left(\left.(A_{t,t+r}^u)^2G'_t\left(Y_{t,r}^{(-u)}\right)\right|\mathcal F_t\right)\right|\right)\nonumber \\
    &\qquad +\mathbb E_x\left(\left|\mathcal S(t+r)^2\sum_{u\in \alives_t}\mathbb E_x\left(\left.(A_{t,t+r}^u)^2G'_t\left(Y_{t,r}^{(-u)}\right)\right|\mathcal F_t\right)-\sigma_f^2W_t\mathbb E_x(G'_t(Y_{t+r})\mid\mathcal F_t)\right|\right).\nonumber
\end{align}

We proceed to bound each of the two terms. To bound the first one, by using a Taylor expansion we can ensure the existence for each $u\in \alives_t$ of a random variable $\xi_{u,t,r}\in \left[\min\{Y_{t+r},Y_{t,r}^{(-u)}\},\max\{Y_{t+r},Y_{t,r}^{(-u)}\}\right]$ such that, almost surely,
\begin{align*}
    Y_{t+r}&G_t(Y_{t+r}) = \mathcal S(t+r) \sum_{u\in \alives_t}A_{t,t+r}^uG_t(Y_{t+r})\\
    &=\mathcal S(t+r)\sum_{u\in \alives_t} A_{t,t+r}^u\left(G_t\left(Y_{t,r}^{(-u)}\right) + \mathcal S(t+r)A_{t,t+r}^uG'_t\left(Y_{t,r}^{(-u)}\right) + \frac{1}{2}\mathcal S(t+r)^2G''_t(\xi_{u,t,r})\right)\\
    &= \mathcal S(t+r)\sum_{u\in \alives_t}A_{t,t+r}^uG_t\left(Y_{t,r}^{(-u)}\right) + \mathcal S(t+r)^2 \sum_{u\in \alives_t}(A_{t,t+r}^u)^2G'_t\left(Y_{t,r}^{(-u)}\right)\\
    &\qquad +\frac{1}{2}\mathcal S(t+r)^3\sum_{u\in \alives_t}(A_{t,t+r}^u)^3G''_t(\xi_{u,t,r}).
\end{align*}
Hence, by taking conditional expectation with respect to $\mathcal F_t$, using that $G_t$ is $\mathcal F_t$-measurable and the conditional independence of $A_{t,t+r}^u$ and $Y_{t,r}^{(-u)}$, we have that almost surely

\begin{align}
    & \left|\mathbb E_x(Y_{t+r}G_t(Y_{t+r})\mid \mathcal F_t)-\mathcal S(t+r)^2\sum_{u\in \alives_t}\mathbb E_x\left(\left.(A_{t,t+r}^u)^2G'_t\left(Y_{t,r}^{(-u)}\right)\right|\mathcal F_t\right)\right| \nonumber\\
    &\qquad \leq \|G_t\|_\infty \mathcal S(t+r)\sum_{u\in \alives_t}\left|\mathbb E_x(A_{t,t+r}^u\mid \mathcal F_t)\right|+\frac{1}{2}\|G''_t\|_\infty \mathcal S(t+r)^3 \sum_{u\in \alives_t} \mathbb E_x\left(|A_{t,t+r}^u|^3\mid \mathcal F_t\right). \label{eq:TCL_bound2}
\end{align}

For the first term in the last expression we can use that, thanks to the branching property, the fact that $(W_t)_{t\geq 0}$ is a closed martingale and Assumption~\ref{assu:eigenelements}, we have that
\[\left|\mathbb E_x(A_{t,t+r}^u\mid\mathcal F_t)\right|=e^{\lambda r}\left|e^{-\lambda r}\mathfrak M_rf(X_t^u)-\gamma (f)h(X_t^u)\right|\leq e^{(\lambda -\rho)r}\|f\|_VV(X_t^u).\]
In addition, if $f=h$, using that $h$ is an eigenfunction for the semigroup one has that in that case
\[\left|\mathbb E_x(A_{t,t+r}^u\mid \mathcal F_t) \right| = 0,\]
and thus
\begin{align*}
\|G_t\|_\infty \mathcal S(t+r)\sum_{u\in \alives_t}\left|\mathbb E_x(A_{t,t+r}^u\mid \mathcal F_t)\right|\leq \mathcal S(t+r)\mathds{1}_{f\neq h}e^{(\lambda-\rho)r}\|f\|_V Z_t(V),    
\end{align*}
where we also used that $\|G_t\|_\infty\leq 1$ almost surely. For the second term in \eqref{eq:TCL_bound2} we have that, since $\EE_x\left(|A_{t,t+r}^u|^3\mid \mathcal F_t\right)=\mathfrak F_r^{(3)}[f](X_t^u)$ (see~\eqref{eq:goticF}) almost surely and by applying Lemma~\ref{lem:fluctuations_order3}, there exists $C_1>0$ such that
\begin{align*}
    \frac{1}{2}\|G''_t\|_\infty \mathcal S(t+r)^3 \sum_{u\in \alives_t} \mathbb E_x\left(|A_{t,t+r}^u|^3\mid \mathcal F_t\right) \leq C_1\dfrac{1}{6}\dfrac{\mathcal S(t+r)^3}{\mathcal S(r)^3}Z_t(V^3),
\end{align*}
where we now used that $\|G_t''\|_\infty\leq 1/3.$
Thus, by taking expectation on \eqref{eq:TCL_bound2} and applying Assumption~\ref{assu:eigenelements}, we have that there exists $C_2>0$ such that 
\begin{align}
    & \mathbb E_x\left(\left|\mathbb E_x(Y_{t+r}G_t(Y_{t+r})\mid \mathcal F_t)-\mathcal S(t+r)^2\sum_{u\in \alives_t}\mathbb E_x\left(\left.(A_{t,t+r}^u)^2G'_t\left(Y_{t,r}^{(-u)}\right)\right|\mathcal F_t\right)\right|\right) \nonumber \\
    &\qquad \leq C_2e^{\lambda t}\left(e^{(\lambda - \rho)r}\mathcal S(t+r)\mathds{1}_{f\neq h}V(x) + \frac{\mathcal S(t+r)^3}{\mathcal S(r)^3}V^3(x)\right) .\label{eq:first_bound}
\end{align}

We now provide a bound for the second term in \eqref{eq:TCL_bound}. Using again the independence of $A_{t,t+r}^u$ and $Y_{t,r}^{(-u)}$ conditionally to $\mathcal F_t$ and recalling that $W_t=e^{-\lambda t}Z_t(h)$, we have
\begin{align*}
   & \mathbb E_x\left(\left|\mathcal S(t+r)^2\sum_{u\in \alives_t}\mathbb E_x\left(\left.(A_{t,t+r}^u)^2G'\left(Y_{t,r}^{(-u)}\right)\right|\mathcal F_t\right)-\sigma_f^2W_t\mathbb E_x(G'(Y_{t+r})\mid\mathcal F_t)\right|\right)\\
   & \qquad \leq \mathbb E_x\left(\left|\mathcal S(t+r)^2 \sum_{u\in \alives_t} \mathbb E_x((A_{t,t+r}^u)^2\mid\mathcal F_t)-\sigma_f^2e^{-\lambda t}\sum_{u\in \alives_t} h(X_t^u)\right|\mathbb E_x\left(\left.\left|G'\left(Y_{t,r}^{(-u)}\right)\right|\ \right| \mathcal F_t\right)\right)\\
   &\qquad + \sigma_f^2e^{-\lambda t}\mathbb E_x\left(\sum_{u\in \alives_t}h(X_t^u) \mathbb E_x\left(\left.\left|G'\left(Y_{t,r}^{(-u)}\right)-G'(Y_{t+r})\right|\ \right| \mathcal F_t\right)\right).
\end{align*}
Using that $\mathbb E_x((A_{t,t+r}^u)^2\mid \mathcal F_t) = \mathfrak F_r^{(2)}[f](X_t^u)$ and that $\|G'_t\|_\infty \leq 1/2$ almost surely, the first expectation is bounded by 
\begin{equation}\label{eq:second_bound1}
\dfrac{1}{2} e^{-\lambda t}\mathbb E_x\left(\sum_{u\in \alives_t}\left|e^{\lambda t}\mathcal S(t+r)^2\mathfrak F_r^{(2)}[f](X_t^u)-\sigma_f^2h(X_t^u)\right|\right),
\end{equation}
whilst by using the Mean Valued Theorem, Jensen's inequality and that $\|G_t''\|_\infty\leq 1/3$ almost surely, we can bound the second expectation by
\begin{align}
\sigma_f^2e^{-\lambda t}&\mathbb E_x\left(\sum_{u\in \alives_t} h(X_t^u) \|G''\|_\infty\mathbb E_x\left(\left.\left|Y_{t,r}^{(-u)}-Y_{t+r}\right|\ \right|\mathcal F_t\right)\right)\nonumber \\
& \leq \dfrac{1}{3} \sigma_f^2e^{-\lambda t}\mathcal S(t+r)\mathbb E_x\left(\sum_{u\in \alives_t}h(X_t^u)\mathfrak F_r^{(1)}[f](X_t^u)\right)\nonumber \\
&\leq \dfrac{1}{3} \sigma_f^2 e^{-\lambda t} \mathcal S(t+r) \mathbb E_x\left(\sum_{u\in \alives_t} V(X_t^u) \left(\mathfrak F_r^{(2)}[f](X_t^u)\right)^{\nf{1}{2}}\right). \label{eq:second_bound2}
\end{align}

In conclusion, we have that the right-hand side of \eqref{eq:Stein_comparison} is bounded by \eqref{eq:first_bound}+\eqref{eq:second_bound1}+\eqref{eq:second_bound2}. We bound each of these three terms within the proof of each theorem separately.

\begin{proof}[Proof of Theorem~\ref{thm:CLT_Martingale}] In this case $f=h$, $\mathcal S(t) = e^{-\nf{\lambda t}{2}}$ and $\sigma_h^2 = \gamma\left(\psi_\infty^{(2)}\right)$. He have the bound in \eqref{eq:first_bound} is equal to 
\[C_2e^{-\nf{\lambda t}{2}}V^3(x),\]
whilst using \eqref{eq:fluctuation_decomp}, we can rewrite \eqref{eq:second_bound1} as
\begin{align*}
    \dfrac{1}{2}e^{-\lambda t}\mathbb E_x&\left(\sum_{u\in \alives_t} \left|e^{-\lambda r}\mathfrak F_r^{(2)}[h](X_t^u) - \sigma_h^2h(X_t^u)\right|\right)\\
    & = \dfrac{1}{2}e^{-\lambda t} \mathbb E_x\left( \sum_{u\in \alives_t} \left|e^{-\lambda r}\moment_r\psi_\infty^{(2)}(X_t^u) - \gamma\left(\psi_\infty^{(2)}\right)h(X_t^u)\right|\right)\\
    &\leq \dfrac{a_2}{2}\|\psi_\infty^{(2)}\|_V^2e^{-\rho r} e^{-\lambda t}\moment_tV^2(x) \\
    &\leq C_3 e^{-\rho r}V^2(x),
\end{align*}
where we used Assumption~\ref{assu:eigenelements} twice and $C_3>0$. Finally, using that
\[\mathfrak F_r^{(2)}[h](x) = \moment_r\psi_\infty^{(2)}(x) \leq C_4 e^{\lambda r}V^2(x),\]
for some $C_4>0$, we have for \eqref{eq:second_bound2},
\begin{align*}
    \dfrac{1}{3}\sigma_h^2 e^{-\nf{3\lambda t}{2}}e^{-\nf{\lambda r}{2}} \mathbb E_x\left(\sum_{u\in \alives_t}V(X_t^u)\left(\mathfrak F_r^{(2)}[h](X_t^u)\right)^{\nf{1}{2}}\right)&\leq \dfrac{C_4}{3}\sigma_h^2 e^{-\nf{3\lambda t}{2}}\moment_tV^2(x)\\
    &\leq C_5e^{-\nf{\lambda t}{2}}V^2(x),
\end{align*}
for some $C_5>0$. Combining these bounds in \eqref{eq:TCL_bound} and using \eqref{eq:Stein_comparison} we conclude that there exists $C_6>0$ such that for all $F \in \mathcal F$ and all $x\in \mathcal X$
\[\sup_{F\in \mathcal F} \left|\mathbb E_x\left(F(Y_{t+r}\right) - \mathbb E_x\left(F(\sigma_{h}\sqrt{W_t}\mathcal Z\right)\right| \leq C_6\left(e^{\nf{-\lambda t}{2}} + e^{-\rho r}\right)\overline{V}^{(3)}(x),\]
with $\overline{V}^{(3)}(x) = \max\limits_{k\leq 3}V^k(x)$. On the other hand, there exists $C_7>0$ such that for $F\in \mathcal F$,
\begin{align*} 
\left|\mathbb E_x\left(F(\sigma_{h}\sqrt{W_t}\mathcal Z)\right) - \mathbb E_x\left(F(\sigma_{h}\sqrt{W}\mathcal Z)\right)\right|&\leq \|F'\|_\infty \sigma_{h}\mathbb E(|\mathcal Z|)\mathbb E_x(|\sqrt{W_t}-\sqrt{W}|)\\
&\leq  \sigma_{h}\mathbb E(|\mathcal Z|) \mathbb E_x\left(\sqrt{|W_t-W|}\right)\\
& \leq  \sigma_{h}\mathbb E(|\mathcal Z|) \mathbb E_x\left((W_t-W)^2\right)^{\nf{1}{4}}\\
&\leq C_7e^{-\nf{\lambda t}{4}}\sqrt{V(x)}. 
\end{align*}

In conclusion, there exists $C>0$ such that for all $r,t>0$ and all $x\in \mathcal X$
\[\textbf{d}\left(Y_{t+r}, \sigma_{h}\sqrt{W}\mathcal Z\right) = \sup_{F\in \mathcal F} \left|\mathbb E_x\left(F(Y_{t+r}\right) - \mathbb E_x\left(F(\sigma_{h}\sqrt{W}\mathcal Z\right)\right| \leq C \left(e^{\nf{-\lambda t}{2}} + e^{-\rho r}\right)V^*(x),\]
with $V^*(x) = \max\{\overline{V}^{(3)}(x),\sqrt{V(x)}\}$. Hence, for all $t>0$ and all $x\in \mathcal X$,
\[\textbf{d}\left(Y_t,\sigma_{h}\sqrt{W}\mathcal Z\right)\leq C \inf_{\substack{t_1,t_2>0\\ t_1+t_2=t}}\left(e^{\nf{-\lambda t_1}{2}} + e^{-\rho t_2}\right)V^*(x)\leq K\exp\left(-\dfrac{\lambda \rho}{2\rho + \lambda}t\right)V^*(x),\]
for some $K>0$, and this ends the proof.
\end{proof}

\begin{remark}\label{rem:weaker_assumption}
    A careful reading of the proof of Theorem~\ref{thm:CLT_Martingale} shows that the result holds with a weaker assumption on the convergence of the semigroup. One can assume that for all $k\leq \kappa$,
    \[\left|e^{-\lambda t}\moment_t f(x) - \gamma(f)h(x)\right|\leq R(t)\|f\|_V^k V^k(x),\]
    where $R:\RR_+\longrightarrow\RR_+$ is a function that verifies $\lim\limits_{t\to +\infty} R(t) = 0$.
\end{remark}
\begin{proof}[Proof of Theorem~\ref{thm:CLT_SBC}] In this case $2\rho >\lambda$ and $\mathcal S(t) = e^{\nf{-\lambda t}{2}}$. The bound in \eqref{eq:first_bound} reads
\[C_2\left(e^{(\nf{\lambda}{2}-\rho)r}e^{\nf{\lambda t}{2}}V(x)+ e^{\nf{-\lambda t}{2}}V^3(x)\right),\]
whilst for the bound on \eqref{eq:second_bound1}, by applying Lemma~\ref{lem:speed_fluctuations2} and Assumption~\ref{assu:eigenelements}, there exists $C_8>0$ such that
\begin{align*}
    \dfrac{1}{2} e^{-\lambda t}\mathbb E_x\left(\sum_{u\in \alives_t}\left|e^{\lambda t}\mathcal S(t+r)^2\mathfrak F_r^{(2)}[f](X_t^u)-\sigma_f^2h(X_t^u)\right|\right)
    &= \dfrac{1}{2} e^{-\lambda t}\mathbb E_x\left(\sum_{u\in \alives_t} \left|e^{-\lambda r} \mathfrak F_r^{(2)}[f](X_t^u)-\sigma_{f,s}^2h(X_t^u)\right|\right)\\
    &\leq C_8e^{(\nf{\lambda}{2}-\rho)r} V^2(x).
\end{align*}

Finally, if we decompose $\mathfrak F_r^{(2)}[f]$ as in \eqref{eq:fluctuation_decomp}, we have that for all $x\in \mathcal X$ by Proposition~\ref{prop:moments} there exists $C_9>0$ such that
\[\mathfrak F_r^{(2)}[f](x) = \momenttwo_r[\widehat f](x) + \gamma(f)^2\moment_r\psi_\infty^{(2)}(x)\leq C_9 e^{\lambda r}V^2(x).\]
and thus we have for \eqref{eq:second_bound2} that there exists $C_{10}>0$ such that
\begin{align*}
    \dfrac{1}{3} \sigma_f^2 e^{-\lambda t} \mathcal S(t+r) \mathbb E_x\left(\sum_{u\in \alives_t} V(X_t^u) \left(\mathfrak F_r^{(2)}[f](X_t^u)\right)^{\nf{1}{2}}\right)& = \dfrac{1}{3}\sigma_{f,s}^2 e^{-\nf{3\lambda t}{2}}e^{-\nf{\lambda r}{2}}\mathbb E_x\left(\sum_{u\in \alives_t} V(X_t^u) \left(\mathfrak F_r^{(2)}[f](X_t^u)\right)^{\nf{1}{2}}\right)\\
    &\leq \dfrac{C_4}{3}\sigma_{f,s}^2e^{-\nf{3\lambda t}{2}}\moment_tV^2(x)\\
    &\leq C_{10}e^{-\nf{\lambda t}{2}}V^2(x),
\end{align*}
where in the last inequality we used Assumption~\ref{assu:eigenelements}.
Proceeding as in the previous proof, there exists $C>0$ such that for all $t>0$ and all $x\in \mathcal X$,
\[\textbf{d}\left(Y_t,\sigma_{f,s}\sqrt{W}\mathcal Z\right)\leq C \inf_{\substack{t_1,t_2>0\\ t_1+t_2=t}}\left(e^{(\nf{\lambda}{2}-\rho)t_1}e^{\nf{\lambda t_2}{2}}+e^{-\nf{\lambda t_2}{2}}\right)V^*(x)\leq K\exp\left(\dfrac{\lambda(\lambda - 2\rho)}{2(\lambda + 2\rho)}t\right)V^*(x),\]
for some $K>0$, and this ends the proof.
\end{proof}
\begin{proof}[Proof of Theorem~\ref{thm:CLT_CBC}] In this case $2\rho = \lambda$ and $\mathcal S(t) = t^{\nf{-1}{2}}e^{\nf{-\lambda t}{2}}$. We will prove that the quantities on \eqref{eq:first_bound} \eqref{eq:second_bound1} and \eqref{eq:second_bound2} converge to zero when $r$ and $t$ go to infinity. A similar reasoning as in the previous proof will allow us to conclude the result.

We start with \eqref{eq:first_bound}, which is bounded by 
\[C_2\left(\dfrac{e^{\nf{\lambda t}{2}}}{\sqrt{t+r}}V(x) + e^{-\nf{-\lambda t}{2}}\sqrt{\dfrac{r}{t+r}}^3V^3(x)\right),\]
which tends to zero when $r\to +\infty$ and then $t\to +\infty$.

To deal with \eqref{eq:second_bound1} and \eqref{eq:second_bound2} we start by pointing out that there exists $C_{11}>0$ such that for all $r>0$ and all $x\in \mathcal X$ 
\begin{align}\label{eq:bound_fluct_CBC}
    \mathfrak F_r^{(2)}[f](x) &= \momenttwo_r[\widehat f](x) + \gamma(f)^2\moment_r\psi_\infty^{(2)}(x) \nonumber\\
    &\leq C_{11}e^{\lambda r}(1+r)V^2(x),
\end{align}
where the equality comes from \eqref{eq:fluctuation_decomp} and the inequality comes from applying Proposition~\ref{prop:moments} in the critical branching case and Assumption~\ref{assu:eigenelements}. We also define 
\[\Phi_r[f](x) =\left|\frac{e^{-\lambda r}}{r}\mathfrak F_r^{(2)}[f](x) - \sigma_{f,c}^2h(x)\right|,\]
which tends to zero when $r$ goes to infinity, according to Corollary~\ref{cor:fluctuations_order2}. Moreover, using \eqref{eq:bound_fluct_CBC} one has that there exists $C_{12}>0$ such that
\begin{equation}\label{eq:upper_bound_Phi}
    \forall r>0, \forall x\in \mathcal X, \Phi_r[f](x) \leq C_{12}V^2(x).
\end{equation}

Now we deal with \eqref{eq:second_bound1}, which is equal to 
\begin{multline*}
    \dfrac{1}{2}e^{-\lambda t} \mathbb E_x\left(\sum_{u\in \alives_t}\left|\dfrac{e^{-\lambda r}}{t+r}\mathfrak F_r^{(2)}[f](X_t^u)-\sigma_{f,c}^2h(X_t^u)\right|\right)  = \dfrac{1}{2}e^{-\lambda t} \mathbb E_x\left(\sum_{u\in \alives_t} \Phi_r[f](X_t^u)\right)\\
    +\dfrac{1}{2}e^{-\lambda t}\left(\dfrac{t}{r(t+r)}\right)e^{-\lambda r}\mathbb E_x\left(\sum_{u\in \alives_t}\mathfrak F_r^{(2)}[f](X_t^u)\right).
\end{multline*}
The first term on the right-hand side tends to zero when $r$ goes to infinity thanks to \eqref{eq:upper_bound_Phi} and the Dominated Convergence Theorem. By \eqref{eq:bound_fluct_CBC} the second term is bounded by 
\[C_{12}\dfrac{1}{2}e^{-\lambda t}\dfrac{t(1+r)}{r(t+r)}V^2(x),\]
which also tends to zero when $r$ goes to infinity.

Finally, we deal with \eqref{eq:second_bound2} by using again \eqref{eq:bound_fluct_CBC} and Assumption~\ref{assu:eigenelements}, obtaining that there exists $C_{13}>0$ such that
\begin{align*}
    \dfrac{1}{3}\sigma_{f,c}^2 \dfrac{e^{-\nf{3\lambda t}{2}}e^{-\nf{\lambda r}{2}}}{\sqrt{t+r}}\mathbb E_x\left(\sum_{u\in \alives_t}V(X_t^u) \left(\mathfrak F_r^{(2)}[f](X_t^u)\right)^{\nf{1}{2}}\right)&\leq \dfrac{C_8}{3}\sigma_{f,c}^2 e^{-\nf{3\lambda t}{2}}\sqrt{\dfrac{1+r}{t+r}}\moment_tV^2(x)\\
    &\leq C_{13}e^{-\nf{\lambda t}{2}}\sqrt{\dfrac{1+r}{t+r}}V^2(x),
\end{align*}
which goes to zero when first $r$ and then $t$ go to infinity.

In conclusion, reasoning as in the proof of Theorem~\ref{thm:CLT_SBC}, we have that 
\[\lim_{t\to +\infty}\lim_{r\to +\infty} \textbf{d}\left(Y_{t+r},\sigma_{f,c}\sqrt{W}\mathcal Z\right)=0,\]
which finishes the proof.
\end{proof}

\section{Applications}\label{sec:applications}

As said in introduction, almost all examples detailed in \cite{bansaye2020ergodic, bansaye2022non, champagnat2023general, del2002stability,ferre2021more,sanchez2023krein, meyn1993stability,canizo2023harris, bertoin2019feynman, hairer2011yet} and on their extensions verifies our main assumptions of semigroup convergence. Consequently, under moment assumptions on the branching mechanisms, our main theorem applies. In particular, up to additionnal assumptions on the moments of the reproduction law, we capture and extend examples of \cite{kesten1966limit,adamczak2015clt,ren2014central,ren2017central} such as branching diffusion processes. Instead of detailing these examples, let us detail here an example where the semigroup does not admits a regular density with respect to the Lebesgue measure.

\subsection{The house of cards model}\label{ex:HoC}

We consider a branching model where each particle has a trait $x\in \mathcal{X}= [0,1]$ whose traits remain constant between branching events. They then branch at a exponential time of parameter $r(x)$; namely the particle is replaced by $k$ new particles with same trait, with some probability $p_k(x)$. We further assume that at rate $1$, each particle survives but gives birth to new individuals whose traits are uniformly distributed over $[0,1]$. See \cite{kingman1978simple} for motivation on this classical model of mutations.

We assume that the number of offspring is bounded; that is there exists $\kappa$ such $\forall k >\kappa, \ p_k(x) =0$. We moreover assume that $x\mapsto p_k(x)$ and $x\mapsto r(x)$ are continuous and then bounded. This assumption trivially ensures the existence of our process since the number of individuals can be bounded by a multi-type branching process. This coupling argument assures our moment assumptions, Assumption~\ref{assu:moments}, with $V\equiv 1$ and $\kappa=4$.

In this case, the mean semigroup $(S_t)$  associated to this dynamics is generated by
$$
\mathcal{A} f(x) = \int_0^1 f(u) du + r(x) \sum_{k\geq 0} (k-1) p_k(x) f(x).
$$
This semigroup was studied in \cite{cloez2024fast} and is, in particular, less regular than in other contexts where the central limit theorem is generally proved. For instance, in contrast to diffusion processes,  this semi-group does not lead to absolutely continuous measure with respect to the Lebesgue measure. However, setting $\alpha(x)=-r(x) \sum_{k\geq 0} (k-1) p_k(x)$, up to a scaling by $e^{-\alpha(0) t}$, this semigroup was studied in \cite{cloez2024fast} (where $a(x) = \alpha(x)-\alpha(0)$). In our setting, \cite[Theorem 1.1]{cloez2024fast} rewrites as follow.

\begin{thm}\label{thm:HoC}
\label{th:cvhc}
Let us assume that for all $x\in (0,1]$,
\begin{equation}\label{eq:HoC_condition_1}
    \sum_{k\geq 0}kp_k(x) < 1-\dfrac{\alpha(0)}{r(x)},
\end{equation}
and that 
\begin{equation}\label{eq:HoC_condition_2}
    \int_0^1 \dfrac{1}{\alpha (x) - \alpha (0)} dx>1.
\end{equation}
then Assumption~\ref{assu:eigenelements} holds with $V\equiv 1$, $\rho = \lambda +\alpha(0)$ ,  for the unique $\lambda>-\alpha(0)$ such that
\[
\int_0^1\frac{1}{\lambda + \alpha(x) } dx =1.
\] 
Moreover,
$h(x)\propto \frac{1}{\lambda +\alpha(x)}$ and $\gamma(dx)= \frac{dx}{\lambda +\alpha(x)}$.

\end{thm}

\begin{proof} The idea is to apply Theorem 1.1 in \cite{cloez2024fast}. We remark that in particular we need $\alpha(0)= 0$, so we start by defining for all $f \in L^1([0,1])$
\[\widehat{\mathcal A}f(x) = \mathcal Af(x) - \alpha (0) = \int\limits_0^1f(u)du - \widehat{\alpha}(x)f(x),\]
with $\widehat{\alpha}(x) = \alpha (x) -\alpha (0)$. Let us verify that $\widehat{\mathcal A}$ verifies the assumptions of Theorem 1.1. We have that, by construction, $\widehat{\alpha}(0)=0$, whilst condition \eqref{eq:HoC_condition_1} ensures that $\widehat{\alpha}(x)>0$ for all $x\in (0,1]$. Hence, condition \eqref{eq:HoC_condition_2} allows is to apply the first part of Theorem 1.1 to $\widehat{\mathcal A}$, which ensures that Assumption~\ref{assu:eigenelements} with $V\equiv 1$. The relation between $\widehat{\mathcal A}$ and $\mathcal A$ allows us to conclude the proof.
    
\end{proof}

As a direct consequence we have the following results.

\begin{corollary}\label{cor:HoC_SBC}
    Assume that \eqref{eq:HoC_condition_1} holds and that 
    \begin{equation}\label{eq:Cor_HoC1}
    \int_0^1\frac{1}{\alpha (x) }dx >1\text{ and }\int_0^1\frac{1}{\alpha (x) - 2\alpha (0)} >1.
    \end{equation}
    Then Theorem~\ref{thm:CLT_SBC} holds.
\end{corollary}

\begin{corollary}\label{cor:HoC_CBC}
    Assume that \eqref{eq:HoC_condition_1} holds, that $\alpha(0)<0$ and that
    \begin{equation}\label{eq:Cor_HoC2}
    \int_0^1\frac{1}{\alpha (x)}dx>1\text{ and }\int_0^1\frac{1}{\alpha (x) -2\alpha (0)}=1.
    \end{equation}
    Then Theorem~\ref{thm:CLT_CBC} holds.
\end{corollary}
\begin{corollary}
    Under the same assumptions that in Theorem~\ref{thm:HoC} we have that if $\lambda > 2\alpha(0)$, Theorem~\ref{thm:CLT_SBC} holds (which is the case for instance, if $\alpha(0)=0$) and if $\lambda = 2\alpha (0)$, then Theorem~\ref{thm:CLT_CBC} holds.
\end{corollary}

\begin{proof}[Proof of Corollary~\ref{cor:HoC_SBC}] Using the mean valued Theorem and \eqref{eq:Cor_HoC1} we have that \eqref{eq:HoC_condition_2} holds and thus we can apply Theorem~\ref{th:cvhc}. In addition, the first condition in \eqref{eq:Cor_HoC1} implies that $\lambda >0$, whilst the second condition implies that $\lambda > -2\alpha (0)$, which is equivalent to $2\rho >\lambda$. Thus, we are in the Small branching case and Theorem~\ref{thm:CLT_SBC} holds.   
\end{proof}
\begin{proof}[Proof of Corollary~\ref{cor:HoC_CBC}] The proof follows similarly to the previous one, except that now we have $\lambda = 2\alpha (0) >0$, and thus we are in the Critical branching case. Thus, Theorem~\ref{thm:CLT_CBC} holds.

\end{proof}

\textbf{Acknowledgements.} The authors have been supported by the ANR project NOLO (ANR-20-CE40-
0015), funded by the French Ministry of Research.

    \bibliographystyle{alpha}
    \bibliography{refs}
\end{document}